\documentclass[12pt]{amsart}
\usepackage{amsmath, amssymb, amscd, bbm, mathrsfs, url, pinlabel, hyperref, verbatim, xargs}
\usepackage[margin=1.25in,centering,letterpaper,dvips]{geometry}
\usepackage{color,dcpic,latexsym,graphicx,epstopdf,comment}
\usepackage[all]{xy}
\usepackage{enumitem}
\usepackage{tikz}
\usetikzlibrary{calc}

\newtheorem {theorem}{Theorem}
\newtheorem {lemma}[theorem]{Lemma}
\newtheorem {proposition}[theorem]{Proposition}

\newtheorem {conjecture}[theorem]{Conjecture}
\newtheorem {definition}[theorem]{Definition}

\newtheorem {question}[theorem]{Question}

\theoremstyle{remark}

\newtheorem {remark}[theorem]{Remark}
\newtheorem {example}[theorem]{Example}

\numberwithin{equation}{section}
\numberwithin{theorem}{section}
\theoremstyle{definition}

\newlist{pcases}{enumerate}{1}
\setlist[pcases]{
  label=\bf{Case~\arabic*:}\protect\thiscase.~,
  ref=\arabic*,
  align=left,
  labelsep=0pt,
  leftmargin=0pt,
  labelwidth=0pt,
  parsep=0pt
}
\newcommand{\case}[1][]{%
  \if\relax\detokenize{#1}\relax
    \def\thiscase{}%
  \else
    \def\thiscase{~#1}%
  \fi
  \item
}

\newcommand{\ZZ}{\mathbb{Z}}
\newcommand{\Z}{\mathbb{Z}}
\newcommand{\R}{\mathbb{R}}
\newcommand{\N}{\mathbb{N}}

\newcommand{\C}{\mathbb{C}}
\newcommand{\F}{\mathbb{F}}
\newcommand{\Q}{\mathbb{Q}}

\newcommand{\RP}{\mathbb{RP}}

\newcommand{\re}{\operatorname{Re}}

\newcommand{\ssm}{\smallsetminus}

\newcommand{\mirror}{\overline}

\newcommand{\pt}{\mathrm{pt}}

\newcommand{\legendre}[2]{ \genfrac{(}{)}{}{}{#1}{#2} }

\tikzset{every picture/.style=thick}
\tikzset{baseline=-\the\dimexpr\fontdimen22\textfont2\relax}

\title{Surgery obstructions and character varieties}
\date{}

\author{Steven Sivek}
\address{Department of Mathematics \\ Imperial College London}
\email{s.sivek@imperial.ac.uk}

\author{Raphael Zentner}
\address{Universit\"{a}t Regensburg}
\email{raphael.zentner@mathematik.uni-regensburg.de}

\begin{document}

\begin{abstract}

We provide infinitely many rational homology 3-spheres with weight-one fundamental groups which do not arise from Dehn surgery on knots in $S^3$.  In contrast with previously known examples, our proofs do not require any gauge theory or Floer homology.  Instead, we make use of the $SU(2)$ character variety of the fundamental group, which for these manifolds is particularly simple: they are all $SU(2)$-cyclic, meaning that every $SU(2)$ representation has cyclic image. Our analysis relies essentially on Gordon-Luecke's classification of half-integral toroidal surgeries on hyperbolic knots, and other classical 3-manifold topology. 
\end{abstract}

\maketitle

\section{Introduction}

Every closed, connected, oriented 3-manifold can be obtained by Dehn surgery on some link in $S^3$, but it is often hard to tell whether a given $Y$ arises by surgery on a knot.  The simplest obstruction is that $\pi_1(Y)$ must have weight one, meaning that it is normally generated by a single element, and in particular $H_1(Y;\Z)$ must be cyclic.

Many examples of irreducible $Y$ with $H_1(Y)$ cyclic are known to not be realizable by Dehn surgery on a knot, but the proofs require much more sophisticated techniques.  These include Casson--Walker invariants \cite{boyer-lines}, instanton gauge theory \cite{auckly}, and Heegaard Floer homology \cite{hom-karakurt-lidman,marengon-kanenobu}.  For many so-called L-spaces, the Heegaard Floer $d$-invariants have been particularly useful, especially in the form of Greene's combinatorial ``changemaker'' obstruction \cite{greene}.  These have led to a complete classification of the spherical manifolds $Y$ which are realizable as Dehn surgery on a nontrivial knot \cite{greene-lens-space,li-ni,gu,prism-1,prism-2,prism-3}.

In this paper, we provide infinitely many new examples of irreducible 3-manifolds $Y$ with weight-one fundamental group which do not arise as Dehn surgery on a knot.  Our methods are new, and most notably, they do not make use of any gauge theory or Floer homology.  Instead, these $Y$ have the following useful property.
\begin{definition} \label{def:su2-cyclic}
A rational homology 3-sphere $Y$ is \emph{$SU(2)$-cyclic} if every representation $\rho: \pi_1(Y) \to SU(2)$ has cyclic image, or equivalently if every $\rho$ has abelian image.
\end{definition}
\noindent This property is especially useful because we can rule out large classes of Dehn surgeries which are known to not produce $SU(2)$-cyclic manifolds, and then in many cases a simple linking form argument suffices to finish the proof.

To be more precise, the 3-manifolds in question are all graph manifolds, of the form
\[ Y(T_{a,b},T_{c,d}) := \big(S^3 \ssm N(T_{a,b})\big) \cup_{T^2} \big(S^3 \ssm N(T_{c,d})\big), \]
in which we glue a meridian of one torus knot exterior to a Seifert fiber of the other and vice versa. Here, $N(T_{a,b})$ denotes a tubular neighborhood of the torus knot $T_{a,b}$.  These manifolds have weight-one fundamental group (Lemma~\ref{lem:weight-one}) and are $SU(2)$-cyclic (Lemma~\ref{lem:splicing-cyclic}).  They arose naturally in our previous study of $SU(2)$-cyclic surgeries on knots \cite{sivek-zentner,sivek-zentner-menagerie}, because they are among the only known examples of $SU(2)$-cyclic 3-manifolds other than lens spaces with cyclic first homology, meaning that they might conceivably result from Dehn surgery on a knot in $S^3$.

Our results include the following special cases, though the methods apply much more generally.  In order to state them, we say that a set $T \subset \N$ has density zero if
\[ \lim_{n \to \infty} \frac{|T \cap \{1,2,\dots,n\}|}{n} = 0. \]

\begin{theorem}[Theorem~\ref{thm:ab-ab}] \label{thm:main-ab-ab}
There is a set $S \subset \N$ of density zero such that if $T_{a,b}$ is a nontrivial torus knot with $a,b>2$ and $ab \not\in S$, then $Y(T_{a,b},T_{-a,b})$ is not Dehn surgery on any knot in $S^3$.
\end{theorem}

\begin{theorem}[Theorem~\ref{thm:ab+ab}] \label{thm:main-ab+ab}
There is a set $S' \subset \N$ of density zero such that if $T_{a,b}$ is a nontrivial torus knot with $|ab| \not\in S'$, then $Y(T_{a,b},T_{a,b})$ is not Dehn surgery on any knot in $S^3$.
\end{theorem}

The sets $S$ and $S'$ in Theorems~\ref{thm:main-ab-ab} and \ref{thm:main-ab+ab} have explicit descriptions:
\begin{align*}
S &= \{n \in \N \mid p\equiv1\!\!\!\!\pmod{8} \textrm{ for all odd primes } p \mid n^2+1\} \\
S' &= \{n \in \N \mid p\not\equiv 3\!\!\!\!\pmod{4} \textrm{ for all primes } p \mid n-1 \textrm{ \textit{or} for all } p \mid n+1\}.
\end{align*}
For concrete examples, one can show that $n \not\in S$ if $n$ is congruent to any of $2$, $3$, $5$, or $6 \pmod{8}$, and that $n \not\in S'$ if $n$ is congruent to $8$ or $10 \pmod{12}$.


Our proofs of Theorems~\ref{thm:main-ab-ab} and \ref{thm:main-ab+ab} rely mostly on classical 3-manifold topology; most of the work goes into ruling out non-integral slopes. We deal individually with the possibilities that our manifolds are surgeries on a torus knot, a satellite knot, or a hyperbolic knot. Most substantially, we use the classification of non-integral toroidal surgeries on hyperbolic knots by Gordon and Luecke, who proved that all such surgeries have at most half-integral slopes \cite{gordon-luecke-atmosthalf}, and that the list of examples constructed by Eudave-Mu\~noz \cite{eudave-munoz-toroidal} was complete \cite{gordon-luecke-nonintegral}. For satellite knots, we use work of Miyazaki--Motegi \cite{miyazaki-motegi-3} to show that in this case the satellite torus compresses in $Y(T_{a,b},T_{c,d})$, and then work of Culler--Gordon--Luecke--Shalen \cite{cgls} to deduce that the knot is an iterated torus knot.  At this point we use the $SU(2)$-character varieties of the $Y(T_{a,b},T_{c,d})$, and specifically the fact that they are as simple as possible, to eliminate this last possibility.  This argument was inspired by our study of $SU(2)$-cyclic manifolds in \cite{sivek-zentner-menagerie}.  Ultimately we conclude the following.

\begin{theorem}[Theorem~\ref{thm:y-nonintegral-surgery}] \label{thm:main-nonintegral}
A 3-manifold $Y = Y(T_{a,b},T_{c,d})$ can be constructed by a surgery of non-integral slope $r$ on a knot $K$ in $S^3$ if and only if
\[ Y \cong \pm Y(T_{l,lm-1}, T_{2,-(2m-1)}) \]
for some integers $l$ and $m$.  In this case $K$ is an Eudave-Mu\~noz knot and $r$ is the corresponding half-integral slope.
\end{theorem}

\begin{remark}
Theorem~\ref{thm:main-nonintegral} implies that the condition $a,b > 2$ in Theorem~\ref{thm:main-ab-ab} is necessary, since $Y(T_{2,2m-1},T_{-2,2m-1})$ arises from half-integral surgery on one of the Eudave-Mu\~noz knots for all $m \geq 2$.
\end{remark}

In order to apply $SU(2)$ character varieties in the proof of Theorem~\ref{thm:main-nonintegral}, we develop a complete classification of $SU(2)$-cyclic surgeries on iterated torus knots.  This builds directly on our work in \cite{sivek-zentner-menagerie}, which yields the classification for surgeries on ordinary torus knots as an immediate corollary.

\begin{theorem}[Theorem~\ref{thm:iterated-cables}] \label{thm:main-iterated-cables}
Let $K$ be an iterated torus knot.  If some nontrivial $r$-surgery on $K$ is $SU(2)$-cyclic, then $K$, $r$, and $S^3_r(K)$ are among the following:
\begin{align*}
K &= T_{p,q}: & r&=pq+\tfrac{1}{m}\ (m \neq 0), & S^3_r(K) &= L(mpq+1,mq^2) \\
K &= T_{p,2\epsilon}: & r&=2\epsilon p, & S^3_r(K) &= L(p,2\epsilon) \# \RP^3 \\
K &= C_{2pq+\epsilon,2}(T_{p,q}): & r &= 4pq+\epsilon, & S^3_r(K) &= L(4pq+\epsilon,4q^2) \\
&& \mathrm{or\ } r &= 4pq+2\epsilon, & S^3_r(K) &= L(2pq+\epsilon,2q^2) \# \RP^3.
\end{align*}
Here $\epsilon$ can be either $+1$ or $-1$.  In particular, if $K$ is an $n$-fold iterated torus knot for some $n \geq 3$ then $K$ has no nontrivial $SU(2)$-cyclic surgeries.
\end{theorem}


While Theorem~\ref{thm:main-iterated-cables} does not present any new examples of $SU(2)$-cyclic manifolds, we do discover some genuinely new ones in Section~\ref{sec:eudave-munoz}.  These are half-integral, toroidal surgeries on some of the Eudave-Mu\~noz knots from \cite{eudave-munoz-toroidal}.  While many of these have the form $Y(T_{a,b},T_{c,d})$ and hence were already known to be $SU(2)$-cyclic, Propositions~\ref{prop:lm0p-classification} and \ref{prop:su2-cyclic-but-not-splicing} describe many more which are provably not of this form.  One such family is given as explicit surgeries on positive braids in Example~\ref{ex:p-nonzero}.

Finally, we note that linking form arguments do not always suffice to prove that $Y(T_{a,b},T_{c,d})$ is not realizable as integer surgery on a knot.  In many cases, we can still prove the desired result using Greene's changemaker obstruction \cite{greene}.  For example:

\begin{theorem}[Theorem~\ref{thm:2-2a+1-2-2b+1}] \label{thm:main-2-odd-2-odd}
Suppose for some positive integers $a \leq b$ that $Y = Y(T_{2,2a+1},T_{2,2b+1})$ is Dehn surgery on a knot in $S^3$.  Then 
\[ (a,b) \in \{ (1,1),(1,2),(1,3),(2,3),(3,3) \}, \]
and the slope is $-n$ where $n = |H_1(Y;\Z)| = 4(2a+1)(2b+1)-1$.
\end{theorem}

We do not know whether the remaining five cases can actually be realized as surgeries.  In any case, Theorem~\ref{thm:main-2-odd-2-odd} leads us to expect that in fact, none of these graph manifolds can be constructed by integer surgeries on knots in $S^3$.  Combining this expectation with Theorem~\ref{thm:main-nonintegral} in the non-integral case, we have the following.

\begin{conjecture} \label{conj:not-surgery}
If $T_{a,b}$ and $T_{c,d}$ are nontrivial torus knots, then $Y(T_{a,b},T_{c,d})$ is the result of $r$-surgery on a knot $K$ in $S^3$ if and only if $K$ is an Eudave-Mu\~noz knot and $r$ is its half-integral toroidal slope.  In particular we must have
\[ Y \cong \pm Y(T_{l,lm-1},T_{2,-(2m-1)}) \]
for some integers $l$ and $m$.
\end{conjecture}

Since the initial appearance of this paper, Duncan McCoy has disproved Conjecture~\ref{conj:not-surgery} by showing $Y(T_{3,5},T_{-3,5})$ arises as $226$-surgery on a knot in $S^3$; his proof is given below as Proposition~\ref{prop:L35}.  The question of precisely which of the other $Y(T_{a,b},T_{c,d})$ can arise by surgeries on knots in $S^3$ remains very interesting.

On the other hand, Proposition~\ref{prop:surgery-link} says that every $Y(T_{a,b},T_{c,d})$ can be constructed by surgery on a 2-component link.  It remains a very interesting open problem to prove, via $SU(2)$ character varieties or otherwise, that some homology 3-sphere is not Dehn surgery on an $n$-component link for any fixed $n \geq 2$ \cite[Problem~3.102(B)]{kirby-list}.

\subsection*{Organization}

In Section~\ref{sec:splicing} we study the topology of the manifolds $Y(T_{a,b},T_{c,d})$.  Section~\ref{sec:eudave-munoz} discusses the toroidal surgeries on the Eudave-Mu\~noz knots, building on work of Ni and Zhang \cite{ni-zhang}, and in particular determines which of these have the form $Y(T_{a,b},T_{c,d})$ and more generally which are $SU(2)$-cyclic.  Section~\ref{sec:cables} is dedicated to the proof of Theorem~\ref{thm:main-iterated-cables}, classfying $SU(2)$-cyclic surgeries on iterated torus knots.  Then in Section~\ref{sec:elementary-obstruction} we prove Theorem~\ref{thm:main-nonintegral} and deduce Theorems~\ref{thm:main-ab-ab} and \ref{thm:main-ab+ab}, and in Section~\ref{sec:changemakers} we review the changemaker condition and prove Theorem~\ref{thm:main-2-odd-2-odd}.

We identify $SU(2)$ with the group of unit quaternions throughout this paper.

\subsection*{Acknowledgments}

The second author is grateful for support by the SFB ``Higher invariants'' (funded by the Deutsche Forschungsgemeinschaft (DFG)) at the University of Regensburg, and for support by a Heisenberg fellowship of the DFG. He is also thankful for hospitality by the FIM at ETH Z\"urich in early 2018 when this work started.  The authors both thank Duncan McCoy for answering Question~\ref{q:L35} and providing the proof of Proposition~\ref{prop:L35}, and the referees for helpful feedback.

\section{Some $SU(2)$-cyclic graph manifolds} \label{sec:splicing}

In this section we review the family of $SU(2)$-cyclic manifolds $Y(T_{a,b},T_{c,d})$, which were originally studied by Motegi \cite{motegi}, and collect some facts about their topology.  Let $T_{a,b}$ and $T_{c,d}$ be two nontrivial torus knots.  Their exteriors $E_{a,b}$ and $E_{c,d}$ are Seifert fibered over the disk, and one can form a graph manifold
\[ Y(T_{a,b}, T_{c,d}) = E_{a,b} \cup_h E_{c,d}, \]
in which the gluing map $h: \partial E_{a,b} \to \partial E_{c,d}$ sends a meridian and a Seifert fiber of $E_{a,b}$ to a Seifert fiber and a meridian of $E_{c,d}$, respectively.  It is easy to see that $Y(T_{a,b},T_{c,d})$ is irreducible, since both $E_{a,b}$ and $E_{c,d}$ are.

In the sequel we will use $\mu_{a,b}$, $\lambda_{a,b}$, and $\sigma_{a,b}$ to denote a meridian, longitude, and Seifert fiber of $E_{a,b}$, and similarly for $\mu_{c,d}$, $\lambda_{c,d}$, and $\sigma_{c,d}$.  In $\pi_1(Y(T_{a,b},T_{c,d}))$ these are related by
\[ \sigma_{a,b} = (\mu_{a,b})^{ab} \lambda_{a,b} \qquad\mathrm{and}\qquad \sigma_{c,d} = (\mu_{c,d})^{cd} \lambda_{c,d}. \]

\begin{lemma} \label{lem:weight-one}
The fundamental group $\pi_1(Y(T_{a,b},T_{c,d}))$ is normally generated by $\mu_{a,b}$ and thus has weight one.  Moreover, $H_1(Y;\Z) \cong \Z/(abcd-1)\Z$.
\end{lemma}

\begin{proof}
The knot complements $E_{a,b}$ and $E_{c,d}$ are normally generated by the meridians $\mu_{a,b}$ and $\mu_{c,d}$ respectively. Since $\pi_1(Y)$ is an amalgamated free product
\[ \pi_1(Y) = \pi_1(E_{a,b}) \ast_{\pi_1(T^2)} \pi_1(E_{c,d}), \]
it is normally generated by $\mu_{a,b}$ and $\mu_{c,d} = \sigma_{a,b}$.  But as an element of $\pi_1(E_{a,b})$, the Seifert fiber $\sigma_{a,b}$ is a product of conjugates of $\mu_{a,b}$, so $\mu_{a,b}$ suffices to normally generate $\pi_1(Y)$.

The homology $H_1(Y)$ is generated by $[\mu_{a,b}]$ and $[\mu_{c,d}]$, with the relations $[\mu_{a,b}] = [\sigma_{c,d}] = cd[\mu_{c,d}]$ and likewise $[\mu_{c,d}] = ab[\mu_{a,b}]$.  This immediately implies that $H_1(Y)$ is in fact generated by $[\mu_{a,b}]$, which has order $|abcd-1|$, as claimed.
\end{proof}

\begin{lemma}[\cite{motegi, zentner-simple, ni-zhang}] \label{lem:splicing-cyclic}
Each $Y = Y(T_{a,b},T_{c,d})$ is $SU(2)$-cyclic.
\end{lemma}

\begin{proof}
Let $\rho: \pi_1(Y) \to SU(2)$ be a representation.  We note that each of $\mu_{a,b}$ and $\mu_{c,d}$ normally generates $\pi_1(Y)$, so if either $\rho(\mu_{a,b})$ or $\rho(\mu_{c,d})$ is $\pm 1$, then the image of $\rho$ is contained in $\{\pm 1\}$.  From now on we assume that neither of these elements is $\pm1$.

The element $\sigma_{a,b}$ is central in $\pi_1(E_{a,b})$, so $\rho(\sigma_{a,b})$ commutes with the image of $\pi_1(E_{a,b})$.  Since $\rho(\sigma_{a,b}) = \rho(\mu_{c,d})$ is not $\pm 1$, this is only possible if $\rho(\pi_1(E_{a,b}))$ lies in the unique $U(1)$ subgroup containing $\rho(\sigma_{a,b})$.  Similarly, $\rho(\pi_1(E_{c,d}))$ lies in the unique $U(1)$ subgroup containing $\rho(\sigma_{c,d})$.  But these subgroups both contain the element $\rho(\mu_{a,b}) = \rho(\sigma_{c,d})$, which is again not $\pm 1$, so they must coincide.  In other words, the image of $\rho$ lies in a single $U(1)$ subgroup and is hence abelian.
\end{proof}

\begin{lemma}[{\cite[Proposition~7.3]{sivek-zentner-menagerie}}] \label{lem:splicing-tori}
The torus $\partial E_{a,b} = \partial E_{c,d}$ is the unique closed, orientable, incompressible surface in $Y(T_{a,b},T_{c,d})$ up to isotopy.
\end{lemma}

\begin{proof}[Proof (sketch)]
Another incompressible surface must intersect the torus $T = \partial E_{a,b} = \partial E_{c,d}$ minimally in some parallel closed curves.  These curves bound a properly embedded, incompressible surface in $E_{a,b}$, so they must be either longitudes or Seifert fibers in $E_{a,b}$, and similarly they must be longitudes or Seifert fibers in $E_{c,d}$.  But the choice of gluing map ensures that these cannot happen simultaneously.
\end{proof}

\begin{proposition} \label{prop:linking-form}
Let $Y = Y(T_{a,b}, T_{c,d})$.  Then the linking form of $Y$ satisfies
\[ \ell_Y(\mu_{a,b}, \mu_{a,b}) = -\frac{cd}{abcd-1} \qquad\mathrm{and}\qquad \ell_Y(\mu_{c,d}, \mu_{c,d}) = -\frac{ab}{abcd-1}. \]
\end{proposition}

\begin{proof}
We prove the first assertion, and the second follows by an identical argument.  Since $\mu_{a,b} = \sigma_{c,d}$ and $\sigma_{a,b} = \mu_{c,d}$, we compute in $H_1(T^2)$ that
\begin{align*}
(abcd-1)\mu_{a,b} &= cd(\sigma_{a,b}-\lambda_{a,b}) - \mu_{a,b} \\
&= cd\mu_{c,d} - cd\lambda_{a,b} - \sigma_{c,d} \\
&= -\lambda_{c,d} - cd\lambda_{a,b}.
\end{align*}
Thus $(abcd-1)\mu_{a,b}$ is homologous in $T^2$ to $-\lambda_{c,d} - cd\lambda_{a,b}$, which bounds the union of one Seifert surface for $T_{c,d}$ in $E_{c,d}$ and $cd$ parallel Seifert surfaces for $T_{a,b}$ in $E_{a,b}$, all with orientation reversed.  The meridian $\mu_{a,b}$ intersects this surface bounded by $(abcd-1)\mu_{a,b}$ with multiplicity $-cd$, so the lemma now follows from the definition of the linking form.
\end{proof}

Our ultimate goal is to show that many of the manifolds $Y(T_{a,b},T_{c,d})$ cannot be realized by Dehn surgery on a single knot in $S^3$.  We note that this is the best possible result, in the following sense.

\begin{proposition} \label{prop:surgery-link}
Each 3-manifold $Y=Y(T_{a,b},T_{c,d})$ can be produced by Dehn surgery on a two-component link in $S^3$.
\end{proposition}

\begin{proof}
Writing $Y = E_{a,b} \cup_h E_{c,d}$, we will find a knot in $E_{c,d}$ on which some Dehn surgery produces $S^1 \times D^2$.  Performing this surgery inside $Y$ results in a Dehn filling of $E_{a,b}$, hence a Dehn surgery on the torus knot $T_{a,b}$, and so we can recover $S^3$ by one more surgery on the core of this filling.

Since $E_{c,d}$ is Seifert fibered over the disk with two singular fibers, we may remove a neighborhood of one singular fiber and identify its boundary with $S^1 \times S^1$ so that each circle $S^1 \times \{\pt\}$ is a regular fiber.  We then glue in a solid torus by identifying $S^1 \times S^1$ with $S^1 \times \partial D^2$, and extend the Seifert fibration across this solid torus as the projection $S^1 \times D^2 \to D^2$.  The result is now Seifert fibered over the disk with one singular fiber, hence it must be $S^1 \times D^2$.
\end{proof}

In fact, some of these manifolds do arise as surgeries on knots in $S^3$.  When the order $|abcd-1|$ of $H_1(Y(T_{a,b},T_{c,d}))$ is odd, then the second author showed that $Y(T_{a,b},T_{c,d})$ is the branched double cover of some knot $L(T_{a,b},T_{c,d})$ in $S^3$ \cite[Theorem~4.14]{zentner-simple}.  If $L(T_{a,b},T_{c,d})$ has unknotting number $1$, the Montesinos trick tells us that $Y(T_{a,b},T_{c,d})$ can be realized as half-integral surgery on a knot $K$ in $S^3$.  In Theorem~\ref{thm:y-nonintegral-surgery} we will determine precisely which of the $Y(T_{a,b},T_{c,d})$ arise in this way.

\begin{example}
The manifold $Y(T_{2,3},T_{2,-3})$ is the branched double cover of $8_{17}$ (see \cite[\S 5]{zentner-simple}), which has unknotting number 1.  One can check that this $SU(2)$-cyclic manifold is $\frac{37}{2}$-surgery on the $(-2,3,7)$-pretzel knot, which will appear in Theorem~\ref{thm:em-classification} as the mirror of the knot $k(2,2,0,0)$.
\end{example}

\section{Surgeries on the Eudave-Mu\~noz knots} \label{sec:eudave-munoz}

Eudave-Mu\~noz \cite{eudave-munoz-toroidal} found an infinite family of hyperbolic knots $k(l,m,n,p)$ in $S^3$ which have toroidal surgeries with half-integer slope; their exteriors were constructed as the branched double covers of tangles, as shown in Figure~\ref{fig:em-knots}.
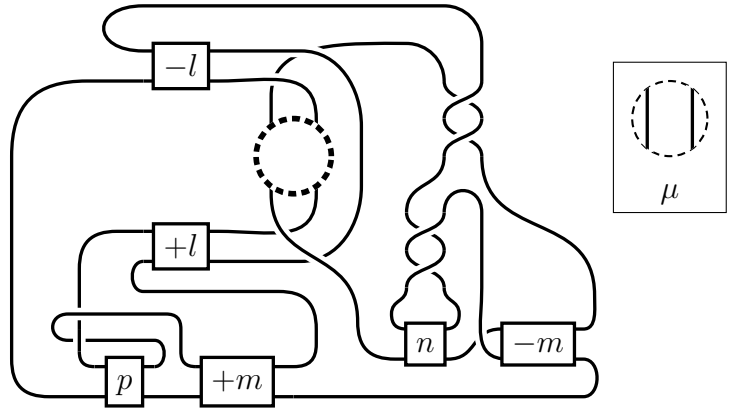
\begin{figure}
\tikzset{link/.style = { white, double = black, line width = 1.75pt, double distance = 1.25pt, looseness=1.25 }}
\begin{tikzpicture}
\draw[link,looseness=0.75] (-0.3,0) -- ++(0,0.8) to[out=90,in=180,looseness=1] (1.5,1.5) to[out=0,in=90,looseness=1] (2,1);
\draw[link] (0.3,0) -- ++(0,0.5) to[out=90,in=0] (-1,1) -- (-2,1); 
\draw[link] (-1,1.4) -- (-2,1.4) to[out=180,in=180,looseness=3] (-2,2) -- (2,2) to[out=0,in=90,looseness=1] (2.5,1.5) -- (2.5,1); 
\draw[link,looseness=1] (2,1) to[out=270,in=90] ++(0.5,-0.5) ++(-0.5,0) to[out=270,in=90] ++(0.5,-0.5);
\draw[link,looseness=1] (2.5,1) to[out=270,in=90] ++(-0.5,-0.5) ++(0.5,0) to[out=270,in=90] ++(-0.5,-0.5);
\draw[link,looseness=1] (-2,-1.4) -- (0.1,-1.4) to[out=0,in=270] (0.9,-0.4) -- (0.9,0.4) to[out=90,in=0] (0.1,1.4) -- (-1.5,1.4); 
\draw[link] (0.3,0) -- ++(0,-0.5) to[out=270,in=0] (-1,-1) -- (-2,-1); 
\draw[link] (0,-3.2) -- (-3.25,-3.2) to[out=180,in=270] (-3.75,-2.7) -- (-3.75,0) to[out=90,in=180] (-2.75,1) -- (-2,1); 
\draw[link] (-1.5,-2.8) -- (0,-2.8) to[out=0,in=270] (0.3,-2.5) -- (0.3,-2.3) to[out=90,in=0] (-0.2,-1.8) -- (-2,-1.8) to[out=180,in=180] (-2,-1.4); 
\draw[link] (-2.65,-2.8) -- (-1.85,-2.8); 
\draw[link] (-2,-1) -- (-2.35,-1) to[out=180,in=90] (-2.85,-1.5) -- (-2.85,-2) -- (-2.85,-2.6) to[out=270,in=180] (-2.65,-2.8); 
\draw[link] (-1.3,-2.8) to[out=180,in=270] (-1.5,-2.6) -- (-1.5,-2.3) to[out=90,in=0] (-1.7,-2.1) -- (-3,-2.1) to[out=180,in=180,looseness=2] (-3,-2.45) -- (-1.85,-2.45) to[out=0,in=0] (-1.85,-2.8); 
\draw[link] (-2.85,-2.6) -- ++(0,0.3); 
\draw[link] (1.75, -2.7) -- (2.05,-2.7) to[out=0,in=180] (2.75,-2.3) -- (3.25,-2.3); 
\draw[link,looseness=1] (1.5,-0.75) to[out=270,in=90] ++(0.5,-0.5) ++(-0.5,0) to[out=270,in=90] ++(0.5,-0.5); 
\draw[link,looseness=1] (2,-0.75) to[out=270,in=90] ++(-0.5,-0.5) ++(0.5,0) to[out=270,in=90] ++(-0.5,-0.5); 
\draw[link] (2,-0.75) to[out=90,in=90,looseness=2] (2.5,-0.75) -- (2.5,-1.75) to[out=270,in=180] (2.85,-2.7) -- (3.25,-2.7); 
\draw[link] (1.5,-0.75) to[out=90,in=270] (2,0); 
\draw[link] (3.25,-2.7) -- (3.85,-2.7) to[out=0,in=0] (3.85,-3.2) -- (0,-3.2); 
\draw[link] (3.25,-2.3) -- (3.85,-2.3) to[out=0,in=270,looseness=1] (4,-1.85) to[out=90,in=270] (2.5,0); 
\draw[link] (1.75,-2.7) -- (1.35,-2.7) to[out=180,in=270] (0.85,-2.2) to[out=90,in=270] (-0.3,-0.5) -- (-0.3,0); 
\draw[link] (2,-1.75) to[out=270,in=90] (2.2,-2.1) to[out=270,in=0] (2,-2.3) -- (1.5,-2.3) to[out=180,in=270] (1.3,-2.1) to[out=90,in=270] (1.5,-1.75); 
\node[draw=black,fill=white,line width=1.25pt,rectangle] at (-1.5,1.2) {$-l$};
\node[draw=black,fill=white,line width=1.25pt,rectangle] at (-1.5,-1.2) {$+l$};
\node[draw=black,fill=white,line width=1.25pt,rectangle] at (-2.25,-3) {$\vphantom{Im}p$};
\node[draw=black,fill=white,line width=1.25pt,rectangle] at (-0.75,-3) {$\vphantom{Ip}{+m}$};
\node[draw=black,fill=white,line width=1.25pt,rectangle] at (1.75,-2.5) {$\vphantom{-m}n$};
\node[draw=black,fill=white,line width=1.25pt,rectangle] at (3.25,-2.5) {$-m$};
\draw[line width=0.075cm,white,fill=white] (0,0) circle (0.5);
\draw[line width=0.075cm,densely dashed,fill=white] (0,0) circle (0.5);
\begin{scope} 
\draw[thin] (4.25,1.25) rectangle (5.75,-0.75);
\node at (5,-0.5) {$\mu$};
\draw[thick,densely dashed,fill=white] (5,0.5) circle (0.5);
\clip (5,0.5) circle (0.5);
\draw[link] (5.3,1) -- (5.3,0) (4.7,1) -- (4.7,0);
\end{scope}
\end{tikzpicture}
\caption{The tangle whose branched double cover is the exterior of the Eudave-Muñoz knot $k(l,m,n,p)$, as in \cite[Figure~2]{eudave-munoz-toroidal}.  Each box contains the indicated number of half-twists, counted with sign; we require that either $n=0$ or $p=0$.  (Filling in the tangle $\mu$ shown at right then turns this tangle into an unknot, whose branched double cover is $S^3$, proving that $k(l,m,n,p)$ is indeed a knot in $S^3$.)} \label{fig:em-knots}
\end{figure}
Gordon and Luecke \cite{gordon-luecke-nonintegral} then proved that any hyperbolic knot $K \subset S^3$ with a non-integral toroidal surgery is one of the $k(l,m,n,p)$, and the unique non-integral toroidal slope $r$ is the one determined by Eudave-Mu\~noz, namely
\begin{equation} \label{eq:r-value}
r = l(2m-1)(1-lm)-\frac{1}{2} + \begin{cases} n(2lm-1)^2, & p=0 \\ p(2lm-l-1)^2, & n=0, \end{cases}
\end{equation}
according to \cite[Proposition~5.3]{eudave-munoz-seifert}.  (By definition each knot $k(l,m,n,p)$ must have either $n=0$ or $p=0$, and a small number of sporadic values of $(l,m,n,p)$ are not allowed because the corresponding $k(l,m,n,p)$ would be a torus knot.)  Each of these knots is also known to be the closure of a positive braid, up to mirroring, and is hence fibered \cite[Proposition~5.6]{eudave-munoz-seifert}.

By building on work of Ni and Zhang \cite{ni-zhang}, who noticed that these surgeries on the knots $k(l,m,0,0)$ came from splicing together torus knot exteriors as in Section~\ref{sec:splicing}, we can completely classify which of the half-integral, toroidal surgeries on Eudave-Mu\~noz knots are $SU(2)$-cyclic, and which of these produce manifolds of the form $Y(T_{a,b},T_{c,d})$.

\begin{theorem} \label{thm:em-classification}
Let $r$ be the half-integral, toroidal slope for $K = k(l,m,n,p)$.  Then $Y = S^3_r(K)$ is $SU(2)$-cyclic if and only if $K$ is one of the following:
\begin{itemize}
\item $K = k(l,m,0,0)$, in which case $Y \cong \pm Y(T_{l,lm-1}, T_{2,-(2m-1)})$;
\item $K = k(l,m,1,0)$, in which case $Y \cong \pm Y(T_{-l,lm-1}, T_{2,2m+1})$;
\item $K = k(l,m,0,1)$, in which case $Y \cong \pm Y(T_{-l,ml-l-1}, T_{2,2m-1})$;
\item $K = k(l,m,0,p)$ with $l$ a multiple of $2p-1$, in which case $Y$ does not have the form $Y(T_{a,b},T_{c,d})$ unless $p$ is $0$ or $1$.
\end{itemize}
\end{theorem}

\begin{proof}
In the case of $k(l,m,0,0)$, Ni and Zhang \cite[p.\ 87]{ni-zhang} proved that the $r$-surgery is $Y(T_{a,b},T_{c,d})$ for some $a,b,c,d$ (to be determined in Lemma~\ref{lem:em-vs-torus-splicing}), which implies that it is $SU(2)$-cyclic by Lemma~\ref{lem:splicing-cyclic}.  The identification for $k(l,m,1,0)$ follows immediately, since this is the mirror of $Y(-l,-m,0,0)$.  In Proposition~\ref{prop:lmn0-classification} we prove that the remaining $k(l,m,n,0)$ are not $SU(2)$-cyclic.

For the knots $k(l,m,0,p)$, Proposition~\ref{prop:lm0p-classification} proves that these $r$-surgeries are $SU(2)$-cyclic if and only if $2p-1$ divides $l$.  We have already seen the cases $p=0$ and $p=1$, since in the latter case $k(l,m,0,1)$ is the mirror of $k(-l,1-m,0,0)$, for which we have already identified $Y$.  In the remaining cases where $2p-1 \mid l$, Proposition~\ref{prop:su2-cyclic-but-not-splicing} proves that $Y$ is not the result of splicing two knot complements together.
\end{proof}

\begin{remark}
We compare our results to those of \cite[Section~5]{ni-zhang} here.  Ni and Zhang classified the $SL(2,\C)$-cyclic $r$-surgeries among the knots $k(l,m,n,0)$, and proved that the $r$-surgeries on $k(l,m,0,p)$ are not $SL(2,\C)$-cyclic when $2p-1 \nmid l$. Their representations factor through triangle groups which have faithful $SL(2,\C)$ representations, but these groups do not always have non-cyclic $SU(2)$ representations, so proving that they are not $SU(2)$-cyclic requires some additional work.  For the surgeries on $k(l,m,0,p)$ with $2p-1 \mid l$, the claims that these are $SU(2)$-cyclic and that they are not $Y(T_{a,b},T_{c,d})$ are new.  In any case, our work relies heavily on \cite{ni-zhang}.
\end{remark}

We begin by describing the geometry of the half-integral, toroidal $r$-surgery on $K=k(l,m,n,p)$.  According to the discussion before \cite[Lemma~4.4]{ni-zhang}, the incompressible torus in $S^3_r(K)$ splits it into two pieces, labeled $X_1$ and $X_2$, each of which is Seifert fibered over $D^2$ with two singular fibers.  There are curves $\mu_1,\sigma_1 \subset \partial X_1$ and $\mu_2,\sigma_2 \subset \partial X_2$, each pair having distance $\Delta(\mu_i,\sigma_i)=1$, whose corresponding Dehn fillings are as follows:
\begin{itemize}
\item $X_1(\mu_1)$ is a lens space of order $|2p-1|$.
\item $X_1(\sigma_1)$ is a connected sum of two lens spaces, of orders
\[ |{-}l| \textrm{ and } |(1-lm)(2p-1)+pl|. \]
\item $X_2(\mu_2)$ is a lens space of order $|2n-1|$.
\item $X_2(\sigma_2)$ is a connected sum of two lens spaces, of orders
\[ 2 \textrm{ and } |2m(2n-1)+1|. \]
\end{itemize}
It follows that $\sigma_1$ and $\sigma_2$ are the Seifert fibers of $X_1$ and $X_2$ respectively, since their Dehn fillings are reducible.  Moreover, when $p = 0$ we have $X_1(\mu_1) = S^3$, so $X_1$ is a knot complement, and since it is Seifert fibered it must be a torus knot complement \cite{moser}.  The unique reducible surgery on $T_{a,b}$ (of slope $ab$) is $L(a,b) \# L(b,a)$, so the description of $X_1(\sigma_1)$ says that if $p=0$ then $X_1$ is the complement of $T_{l,\pm(lm-1)}$.  Similarly, if $n=0$ then $X_2$ must be the complement of $T_{2,\pm(2m-1)}$.

We can now explicitly identify the half-integer toroidal surgeries on the knots $k(l,m,0,0)$.  We note that according to \cite{eudave-munoz-toroidal}, if $n=p=0$ then we must have $l \neq -1,0,1$ and $m \neq 0,1$, and also $(l,m) \neq (-2,-1)$.  (For the remaining $k(l,m,0,0)$, it is easy to check that $r<0$.)

\begin{lemma} \label{lem:em-vs-torus-splicing}
Let $r$ be the slope of the half-integral, toroidal, $SU(2)$-cyclic surgery on $k(l,m,0,0)$.  Then
\[ S^3_r(k(l,m,0,0)) \cong \pm Y(T_{l,lm-1}, T_{2,-(2m-1)}). \]
\end{lemma}

\begin{proof}
Let $Y = S^3_r(k(l,m,0,0)) = X_1 \cup_{T^2} X_2$ as above.  Then we have seen that $X_1$ and $X_2$ are the complements of $T_{l,\pm(lm-1)}$ and $T_{2,\pm(2m-1)}$ for some choices of sign, with $\mu_i$ and $\sigma_i$ the corresponding meridians and Seifert fibers.  Since $\mu_1$ is glued to $\sigma_2$ and $\sigma_1$ to $\mu_2$, we have $Y = Y(T_{l,\epsilon_1(lm-1)}, T_{2,\epsilon_2(2m-1)})$ for some signs $\epsilon_1,\epsilon_2 = \pm 1$ by definition, and then
\[ |H_1(Y(T_{l,\epsilon_1(lm-1)},T_{2,\epsilon_2(2m-1)}))| = |l\cdot\epsilon_1(lm-1) \cdot 2\cdot\epsilon_2(2m-1) - 1|. \]
But equation~\eqref{eq:r-value} says that $|H_1(Y)| = |2l(2m-1)(1-lm)-1|$, and since $l\neq 0$ and $lm\neq 1$, this agrees with $|H_1(Y)|$ if and only if $\epsilon_1\epsilon_2 = -1$.  Thus up to reversing orientation, we can take $\epsilon_1 = 1$ and $\epsilon_2 = -1$, proving the lemma.
\end{proof}

\begin{proposition} \label{prop:lmn0-classification}
Let $r$ be the slope of the half-integral, toroidal surgery on $K=k(l,m,n,0)$.  Then $Y = S^3_r(K)$ is $SU(2)$-cyclic if and only if $n$ is either 0 or 1.
\end{proposition}

\begin{proof}
The mirror of $k(l,m,n,0)$ is $k(-l,-m,1-n,0)$, and the case $n=0$ (and hence its mirror $n=1$) is \cite[Corollary~5.3]{ni-zhang} (or Lemmas~\ref{lem:splicing-cyclic} and \ref{lem:em-vs-torus-splicing}), so we can assume $n \neq 0,1$ and then our goal is to show that $Y$ is not $SU(2)$-cyclic.  We observe that $|2n-1| \geq 3$, and that $X_1$ is the complement of $T_{l,\pm(lm-1)}$ since $p=0$.

The map $i_*: H_1(\partial X_2) \to H_1(X_2)$ has rank 1 by the ``half lives, half dies'' principle (see e.g.\ \cite[Lemma~3.5]{hatcher-3}), so we let $x \in H_1(\partial X_2;\Z)$ generate the kernel; then $x$ is uniquely a positive integral multiple of a primitive element $\lambda \in H_1(\partial X_2;\Z)$, with $i_*\lambda$ having finite order.  According to \cite[Lemma~3.2]{watson}, if $T$ is the torsion subgroup of $H_1(X_2;\Z)$ then we have
\[ |H_1(X_2(\alpha);\Z)| = |T| \cdot \operatorname{ord}_T(i_*\lambda) \cdot \Delta(\alpha,\lambda) \]
for all slopes $\alpha$.  But then
\[ |H_1(X_2(\mu_2))| = |2n-1| \quad\text{and}\quad |H_1(X_2(\sigma_2))| = 2|2m(2n-1)+1| \]
are relatively prime, so $T$ is trivial, hence $i_*\lambda = 0$.  This means that $x=\lambda$, and since $x$ is therefore primitive, we can extend it to an integral basis $x,y$ of $H_1(\partial X_2)$.

We now write
\[ [\mu_2] = ax + by, \qquad [\sigma_2] = cx+dy \]
for some integers $a,b,c,d$, and we get
\[ H_1(X_2(\mu_2)) = H_1(X_2) / \langle by \rangle, \qquad H_1(X_2(\sigma_2)) = H_1(X_2) / \langle dy \rangle. \]
Again, the orders of these groups are relatively prime, and since $H_1(X_2)/\langle y\rangle$ is a quotient of both groups its order must divide both of their orders, so it must be the trivial group: in other words, $H_1(X_2) \cong \langle y\rangle$.  In particular, we have $|b| = |H_1(X_2(\mu_2))| = |2n-1| \geq 3$.

Now $\pi_1(Y) \cong \pi_1(X_1) \ast_{\pi_1(T^2)} \pi_1(X_2)$ surjects onto the group
\[ \pi_1(X_1) \ast_{\pi_1(T^2)} H_1(X_2) \cong \pi_1(X_1) / \langle x \rangle \cong \pi_1(X_1(\gamma)), \]
where $\gamma \subset \partial X_1$ is the curve which is glued to $x \subset \partial X_2$.  We have $\Delta(\sigma_1,\gamma) = \Delta(\mu_2,x) = |b|\Delta(y,x) \geq 3$, so $X_1(\gamma)$ is the result of Dehn surgery on $T_{l,\pm(lm-1)} \subset S^3$ along a curve of distance $\delta = \Delta(\sigma_1,\gamma) > 1$ from the Seifert fiber.  Then $X_1(\gamma)$ is Seifert fibered with base orbifold $S^2(|l|,|lm-1|,\delta)$ \cite{moser}, so by \cite{sivek-zentner-menagerie} it is not $SU(2)$-cyclic.  Composing the surjection $\pi_1(Y) \to \pi_1(X_1(\gamma))$ with an irreducible $SU(2)$ representation of the latter proves that $Y$ is not $SU(2)$-cyclic either.
\end{proof}

Understanding the toroidal surgeries on the knots $k(l,m,0,p)$ requires a little more effort.  In this case, since $n=0$, we recall that $X_2$ is the complement of $T_{2,\pm(2m-1)}$.

\begin{lemma} \label{lem:lm0p-reducible}
Let $Y = X_1 \cup_{T^2} X_2$ be the half-integral, toroidal surgery on $k(l,m,0,p)$.  If $\rho: \pi_1(Y) \to SU(2)$ is irreducible, then $\rho|_{\pi_1(X_1)}$ has abelian image and $\rho|_{\pi_1(X_2)}$ is irreducible.  In this case we have $\rho(\mu_1) = \rho(\sigma_2) = -1$ and $\rho(\mu_2)=\rho(\sigma_1) \not\in \{\pm 1\}$.
\end{lemma}

\begin{proof}
Let $\rho_i = \rho|_{\pi_1(X_i)}$ for $i=1,2$, and write $X_2 = S^3 \ssm T_{2,k}$ where $k=\pm(2m-1)$.  

Suppose first that $\rho_1$ is irreducible.  Since the Seifert fiber $\sigma_1$ is central in $\pi_1(X_1)$, its image $\rho_1(\sigma_1)$ must belong to the center $\{\pm 1\}$ of $SU(2)$.  In this case $\rho(\mu_2) = \rho(\sigma_1) = \pm 1$, and $\pi_1(X_2) = \pi_1(S^3 \ssm T_{2,k})$ is normally generated by the meridian $\mu_2$, so the image of $\rho_2$ is contained in $\{\pm 1\}$, meaning that it factors through $H_1(X_2) = \langle \mu_2 \rangle$.  The homology class of the Seifert fiber of $X_2$ is $[\sigma_2] = 2k[\mu_2] + [\lambda_2]$, where the longitude $\lambda_2$ is nullhomologous, so $\rho_2(\sigma_2) = \rho_2(\mu_2)^{2k} = 1$, and hence
\[ \rho_1(\mu_1) = \rho(\mu_1) = \rho(\sigma_2) = \rho_2(\sigma_2) = 1. \]
But then $\rho_1$ factors through the cyclic group $\pi_1(X_1) / \langle \mu_1 \rangle = \pi_1(X_1(\mu_1)) = \Z/(2p-1)\Z$, since $X_1(\mu_1)$ is a lens space of order $|2p-1|$, and we have a contradiction.

Now suppose instead that $\rho_2$ is reducible.  The longitude $\lambda_2$ satisfies $\Delta(\mu_2,\lambda_2)=1$, hence it is glued to a curve $\gamma \subset \partial X_1$ with $\Delta(\sigma_1,\gamma)=1$.  Since $\pi_1(X_2)$ is normally generated by $\mu_2$, the representation $\rho$ factors through
\[ \pi_1(X_1) \ast_{\pi_1(T^2)} H_1(X_2) = \pi_1(X_1) \ast_{\pi_1(T^2)} \langle \mu_2\rangle, \]
and since $\ker(H_1(\partial X_2) \to H_1(X_2))$ is generated by $\lambda_2$, we can identify this last group as $\pi_1(X_1) / \langle \gamma \rangle = \pi_1(X_1(\gamma))$.  Since $\gamma \neq \sigma_1$, the Seifert fibration on $X_1$ extends to $X_1(\gamma)$ by adding a singular fiber of order $\Delta(\sigma_1,\gamma)=1$ (see \cite{heil}); then $X_1(\gamma)$ has three singular fibers, one of which has order 1, and so it is a lens space.  But now we have shown that $\rho$ factors through a cyclic group, and this contradicts the assumption that it was irreducible.  So $\rho_2$ must be irreducible, and then $\rho_2(\mu_2) \not\in \{\pm 1\}$ follows immediately.

Finally, since $\rho_2$ is irreducible, it must send the central element $\sigma_2 \in \pi_1(X_2)$ to the center $\{\pm 1\}$ of $SU(2)$.  The fundamental group of $S^3 \ssm T(2,k)$ has a presentation
\[ \pi_1(X_2) = \langle x,y \mid x^2 = y^k \rangle, \]
in which $\sigma_2 = x^2 = y^k$.  If $\rho_2(\sigma_2) = 1$ then $\rho_2(x)^2 = 1$, which can only happen if $\rho_2(x) = \pm 1$.  But then $\rho_2$ would have abelian image, contradicting its irreducibility, and so we conclude that $\rho_2(\sigma_2) = -1$.
\end{proof}

\begin{proposition} \label{prop:lm0p-classification}
The half-integral, toroidal surgery on $k(l,m,0,p)$ is $SU(2)$-cyclic if and only if $2p-1$ divides $l$.
\end{proposition}

\begin{proof}
Writing the surgered manifold as $Y = X_1 \cup_{T^2} X_2$, we know from Lemma~\ref{lem:lm0p-reducible} that any irreducible representation $\rho: \pi_1(Y) \to SU(2)$ must restrict to a representation $\rho_1: \pi_1(X_1) \to SU(2)$ with abelian image, satisfying $\rho(\mu_1) = -1$ and $\rho(\sigma_1) \neq \pm 1$.  Since the image is abelian, $\rho_1$ factors through $H_1(X_1)$, and we can arrange up to conjugacy that its image lies in the 1-parameter subgroup $\{e^{it}\}$.

We first describe $H_1(X_1)$.  The Seifert fibration on $X_1$ comes with singular fibers of order $\alpha_1 = |{-}l|$ and $\alpha_2 = |(1-lm)(2p-1)+pl|$, which are determined as the orders of the lens space summands of $X_1(\sigma_1)$.  As we have already seen, the filling $X_1(\mu_1)$ is Seifert fibered over $S^2$ with an additional fiber of order $\Delta(\mu_1,\sigma_1) = 1$.  Following \cite{jankins-neumann}, if its Seifert invariants are $(\alpha_1,\beta_1)$, $(\alpha_2,\beta_2)$, and $(\alpha_3,\beta_3)=(1,k)$ for some integer $k$, then we have a presentation
\[ \pi_1(X_1(\mu_1)) = \langle c_1,c_2,c_3,h \mid [h,c_i]=c_i^{\alpha_i}h^{\beta_i}=1\,\,\forall i, c_1c_2c_3=1 \rangle, \]
from which $H_1(X_1(\mu_1)) \cong \Z/(2p-1)\Z$ is generated by classes $[c_1],[c_2],[c_3],[h]$ satisfying several relations.  We use the relation $[c_3]+k[h]=0$ to eliminate $[c_3]$, and then $H_1(X_1(\mu_1)) = \Z/(2p-1)\Z$ is presented by the matrix
\[ M = \begin{pmatrix} \alpha_1 & 0 & \beta_1 \\ 0 & \alpha_2 & \beta_2 \\ -1 & -1 & k \end{pmatrix} \]
corresponding to generators $[c_1]$, $[c_2]$, and $[h] = [\sigma_1]$.  Thus $|{\det}(M)| = |2p-1|$.

Now, we determine the possible values of $\rho_1$ on the generators of $H_1(X_1)$.  This amounts to finding constants $\theta_1,\theta_2,\varphi \in \R/2\pi\Z$ such that
\[ \rho_1(c_j) = e^{i\theta_j}, \qquad \rho_1(\sigma_1) = e^{i\varphi}. \]
Aside from $h$ being central, the only relations in $\pi_1(X_1)$ are $c_1^{\alpha_1} h^{\beta_1} = c_2^{\alpha_2} h^{\beta_2} = 1$, which are equivalent here to $\alpha_j\theta_j + \beta_j\varphi \equiv 0 \pmod{2\pi}$ for $j=1,2$.
From the relation $\mu_1 = (c_1c_2)^{-1}h^k$ in $\pi_1(X_1)$, we have
\[ \rho_1(\mu_1) = e^{i(-\theta_1-\theta_2+k\varphi)}, \]
so $\rho_1(\mu_1)=-1$ imposes the condition $-\theta_1-\theta_2+k\varphi \equiv \pi \pmod{2\pi}$.  Thus $\theta_1,\theta_2,\varphi$ define a representation $H_1(X_1) \to SU(2)$ if and only if for some integers $v_1,v_2,w$, we have
\[ \begin{pmatrix} \alpha_1 & 0 & \beta_1 \\ 0 & \alpha_2 & \beta_2 \\ -1 & -1 & k \end{pmatrix} \begin{pmatrix} \theta_1 \\ \theta_2 \\ \varphi \end{pmatrix} = \begin{pmatrix} 2\pi v_1 \\ 2\pi v_2 \\ \pi + 2\pi w\end{pmatrix}, \]
or equivalently if and only if 
\[ \begin{pmatrix} \theta_1 \\ \theta_2 \\ \varphi \end{pmatrix} = \frac{1}{\det(M)} \begin{pmatrix} k\alpha_2+\beta_2 & -\beta_1 & -\beta_1\alpha_2 \\ -\beta_2 & \alpha_1k + \beta_1 & -\alpha_1\beta_2 \\ \alpha_2 & \alpha_1 & \alpha_1\alpha_2 \end{pmatrix} \begin{pmatrix} 2\pi v_1 \\ 2\pi v_2 \\ \pi + 2\pi w\end{pmatrix}, \]
in which case $\varphi = \frac{\pi}{\det(M)}\left(2\alpha_2v_1 + 2\alpha_1v_2 + (1+2w)\alpha_1\alpha_2\right)$.

Suppose that $2p-1$ divides $l$.  Then $\alpha_1,\alpha_2,v_1,v_2,w$ are all integers, and $\det(M) = \pm(2p-1)$ divides both $\alpha_1 = |{-}l|$ and $\alpha_2 = |(1-lm)(2p-1)+pl|$, so $\varphi$ is an integral multiple of $\pi$.  But then $\rho_1(\sigma_1) = \pm 1$, and this is impossible if $\rho_1$ is the restriction of an irreducible representation $\pi_1(Y) \to SU(2)$.  We conclude that $Y$ is $SU(2)$-cyclic whenever $2p-1$ divides $l$.

On the other hand, if $2p-1$ does not divide $l$ then we can take
\[ (v_1,v_2,w) = (0,\tfrac{1}{2}(q-\alpha_2),0) \]
for some integer $q \equiv \alpha_2 \pmod{2}$, and this gives a solution with $\varphi = \frac{\pi \alpha_1q}{\det(M)}$.  Write
\[ A = \frac{\alpha_1}{\gcd(l,2p-1)} \qquad\mathrm{and}\qquad D = \frac{\det(M)}{\gcd(l,2p-1)}; \]
then $A$ and $D$ are relatively prime, with $D$ odd since it divides $2p-1$ and $|D| \neq 1$ since $\gcd(l,2p-1) \neq \pm(2p-1)$.  Since $A$ is invertible mod $D$ and $\frac{D+1}{2}$ is an integer, we can take $q \equiv \left(\frac{D+1}{2}\right)A^{-1} \pmod{D}$: there are two such integers between $1$ and $2|D|$ inclusive, and they differ by $D$, so we take whichever value has the same parity as $\alpha_2$.  Then
\[ \varphi = \frac{\pi Aq}{D} \equiv \pi\left(\frac{1}{2} + \frac{1}{2D}\right) \pmod{\pi}, \]
so $\varphi$ is not an integer multiple of $\pi$.  In fact, since $|D| \geq 3$, we have $\frac{\pi}{3} \leq \varphi \leq \frac{2\pi}{3}$ modulo $\pi$.

The abelian representation $\rho_1: \pi_1(X_1) \to SU(2)$ which we have just constructed extends to $\rho: \pi_1(Y) \to SU(2)$ if and only if there is an irreducible $\rho_2: \pi_2(X_2) \to SU(2)$ which agrees with $\rho_1$ on their boundary torus: this means that
\[ \rho_2(\mu_2) = \rho_1(\sigma_1) = e^{i\varphi} \qquad\mathrm{and}\qquad \rho_2(\sigma_2) = \rho_1(\mu_1) = -1. \]
Since $X_2$ is the complement of the $(2,k)$-torus knot where $k = \pm(2m-1)$, a straightforward exercise shows this is possible whenever $\frac{\pi}{2|k|} < \varphi < \pi - \frac{\pi}{2|k|} \pmod{\pi}$.  But $n=0$ implies that $m\neq 0,1$, so $|k| \geq 3$, and then the range $\frac{\pi}{2|k|} < \varphi < \pi-\frac{\pi}{2|k|}$ contains the open interval $\left(\frac{\pi}{6}, \frac{5\pi}{6}\right)$ and hence the closed interval $[\frac{\pi}{3},\frac{2\pi}{3}]$ in which our actual value of $\varphi$ lies.  This means that $\rho_1|_{\partial X_1}$ does extend to an irreducible representation $\rho_2$ on $X_2$, and hence $\rho_1$ extends to an irreducible $\rho$ on all of $Y$, as desired.
\end{proof}

\begin{proposition} \label{prop:su2-cyclic-but-not-splicing}
Let $K = k(l,m,0,p)$, where $2p-1$ divides $l$, and let $r$ be the slope of the half-integral, toroidal surgery on $K$.  If $p$ is neither 0 nor 1, then $Y = S^3_r(K)$ is not one of the manifolds $Y(T_{a,b},T_{c,d})$.
\end{proposition}

\begin{proof}
We decompose $Y = X_1 \cup_{T^2} X_2$ along an incompressible torus as usual.  If $Y \cong Y(T_{a,b},T_{c,d})$ then this torus is the unique closed, orientable, incompressible surface in $Y$ by Lemma~\ref{lem:splicing-tori}, so $X_1$ and $X_2$ must be the exteriors of $T_{a,b}$ and $T_{c,d}$ in some order.  We write $X_1 = S^3 \ssm N(T_{a,b})$ without loss of generality, and then its Dehn filling along a Seifert fiber is a connected sum of two lens spaces, whose orders $|a|$ and $|b|$ are relatively prime.  On the other hand, we have seen that $X_1(\sigma_1)$ is actually a sum of lens spaces of orders
\[ |{-}l| \quad \mathrm{and} \quad |(1-lm)(2p-1)+pl|, \]
and by hypothesis these are both multiples of $2p-1$, so they can only be relatively prime if $2p-1 = \pm 1$, or equivalently if $p$ is either $0$ or $1$.
\end{proof}

\begin{example} \label{ex:p-nonzero}
The half-integral, toroidal surgeries on Eudave-Mu\~noz knots $k(l,m,0,p)$, where $2p-1$ divides $l$ and is not equal to $\pm1$, are $SU(2)$-cyclic but not of the form $Y(T_{a,b},T_{c,d})$ by Propositions~\ref{prop:lm0p-classification} and \ref{prop:su2-cyclic-but-not-splicing}.  Thus if we let $K_q$ be the mirror of $k(2q+1,-1,0,-q)$ for any $q \geq 1$, then some straightforward manipulation of the rightmost portion of \cite[Figure~2]{ni-zhang} shows that $K_q$ is a twisted torus knot, namely
\[ K_q = T(6q+4,6q^2+6q+1,2q+2,2), \]
which by definition is the closure of the positive braid
\[ (\sigma_{6q+3}\sigma_{6q+2}\dots\sigma_2\sigma_1)^{6q^2+6q+1}(\sigma_{2q+1}\sigma_{2q}\dots\sigma_2\sigma_1)^2 \in B_{6q+4}. \]
The half-integral, toroidal, $SU(2)$-cyclic surgery on $K_q$ has slope
\[ r_q = \tfrac{1}{2}(72q^3+120q^2+68q+13) = \frac{2q+1}{2}(36q^2+42q+13) \]
according to Equation~\eqref{eq:r-value}.
\end{example}

\section{$SU(2)$-cyclic surgeries on iterated cables} \label{sec:cables}

In this section we use results from \cite[\S 10]{sivek-zentner} and \cite{sivek-zentner-menagerie} to completely determine the set of $SU(2)$-cyclic surgeries on iterated cables of torus knots.  We will use the following conventions for cables and lens spaces.  If $K \subset S^3$ is a knot, with meridian $\mu$ and longitude $\lambda$ in $\partial N(K)$, then the $(m,n)$-cable $C_{m,n}(K)$ is an embedded curve in $\partial N(K)$ in the homology class $m[\mu]+n[\lambda]$; we always assume that $n \geq 2$, since the $(m,1)$-cable is isotopic to $K$.  We also write $L(p,q)$ for the lens space obtained by $\frac{p}{q}$-surgery on the unknot in $S^3$.

\begin{theorem} \label{thm:iterated-cables}
Let $K = C_{p_n,q_n}(C_{p_{n-1},q_{n-1}}(\dots(C_{p_2,q_2}(T_{p_1,q_1}))\dots))$ be an iterated torus knot, with $\gcd(p_i,q_i)=1$ and $q_i \geq 2$ for all $i$, and $|p_1| \neq 1$.  If some nontrivial $r$-surgery on $K$ is $SU(2)$-cyclic, then $K$, $r$, and $S^3_r(K)$ are among the following:
\begin{align*}
K &= T_{p,q}: & r&=pq+\tfrac{1}{m}\ (m \neq 0), & S^3_r(K) &= L(mpq+1,mq^2) \\
K &= T_{p,2\epsilon}: & r&=2\epsilon p, & S^3_r(K) &= L(p,2\epsilon) \# \RP^3 \\
K &= C_{2pq+\epsilon,2}(T_{p,q}): & r &= 4pq+\epsilon, & S^3_r(K) &= L(4pq+\epsilon,4q^2) \\
&& \mathrm{or\ } r &= 4pq+2\epsilon, & S^3_r(K) &= L(2pq+\epsilon,2q^2) \# \RP^3.
\end{align*}
Here $\epsilon$ can be either $+1$ or $-1$.  In particular, if $n \geq 3$ then $K$ has no nontrivial $SU(2)$-cyclic surgeries.
\end{theorem}

We remark that the $(2pq+\epsilon,2)$-cables of $T_{p,q}$ have certainly appeared in the literature before.  Notably, the $4pq$-surgeries on these cables belong to a short list of toroidal Seifert fibered surgeries on non-hyperbolic knots \cite{miyazaki-motegi}, having base $\RP^2$ and two singular fibers.  Also, Greene \cite{greene} classified the Dehn surgeries which produce a connected sum of lens spaces, and the examples in Theorem~\ref{thm:iterated-cables} are precisely those where one summand is $\RP^3$.

We break the proof of Theorem~\ref{thm:iterated-cables} into several cases.  The case $n=1$, in which $K$ is a torus knot, is Proposition~\ref{prop:tpq-slopes}.  For $n=2$, with $K = C_{p_2,q_2}(T_{p,q})$, the case $|p_2| \geq 2$ is proved in Proposition~\ref{prop:m-not-1} as an easy corollary of some results from \cite{sivek-zentner}, and the case $p_2 = \pm 1$ is addressed separately in Proposition~\ref{prop:m-1}.  We then complete the proof by showing inductively that there are no $SU(2)$-cyclic surgeries for $n \geq 3$.

For the case $n=1$, we use the following input from \cite{sivek-zentner-menagerie}.  A manifold is said to be \emph{$SU(2)$-abelian} if every representation $\pi_1(Y) \to SU(2)$ has abelian image; this is equivalent to being $SU(2)$-cyclic for rational homology spheres but weaker in general.

\begin{theorem}[\cite{sivek-zentner-menagerie}] \label{thm:small-seifert-fibered}
Let $Y$ be a small Seifert fibered 3-manifold, with base $S^2$ and at most three singular fibers, which is not a lens space.  Then $Y$ is $SU(2)$-abelian if and only if either
\begin{itemize}
\item the base orbifold is $S^2(2,4,4)$, or
\item the base orbifold is $S^2(3,3,3)$ and $|H_1(Y;\Z)|$ is even or infinite.
\end{itemize}
\end{theorem}

\begin{proposition} \label{prop:tpq-slopes}
The nontrivial $SU(2)$-cyclic surgery slopes for the $(p,q)$-torus knot $T_{p,q}$ are precisely $pq + \frac{1}{n}$ for any nonzero integer $n$, as well as $pq$ if either $p$ or $q$ is $\pm 2$.  These surgeries are lens spaces (for $pq+\frac{1}{n}$) or a connected sum of $\RP^3$ and a lens space (for $pq$, if either $p$ or $q$ is $\pm 2$).
\end{proposition}

\begin{proof}
We consider the $\frac{m}{n}$-surgery on $T_{p,q}$, and let $\Delta = \Delta(\frac{m}{n},pq) = |pqn-m|$.  Moser \cite{moser} characterized these surgeries in terms of $\Delta$: if $\Delta > 0$ then the surgery is Seifert fibered with base $S^2(|p|,|q|,\Delta)$.  This is a lens space when $\Delta=1$, i.e.\ when $\frac{m}{n} = pq + \frac{1}{n}$, but when $\Delta \geq 2$ it is not a lens space and Theorem~\ref{thm:small-seifert-fibered} says it is not $SU(2)$-cyclic.  (Indeed, $|p|$ and $|q|$ are relatively prime, in contrast with the orders of the orbifold points of $S^2(2,4,4)$ or of $S^2(3,3,3)$.)

In the remaining case $\Delta=0$, meaning that $\frac{m}{n} = pq$, Moser showed that the $\frac{m}{n}$-surgery on $T_{p,q}$ is $L(p,q) \# L(q,p)$.  This has fundamental group
\[ \langle x,y \mid x^{|p|} = y^{|q|} = 1 \rangle, \]
so if $|p|, |q| \geq 3$ then we can define a non-cyclic $SU(2)$ representation by $x \mapsto e^{2\pi i/p}$ and $y \mapsto e^{2\pi j/q}$.  If this is not the case, then taking $p = \pm 2$ without loss of generality, we have $L(p,q) \# L(q,p) = \RP^3 \# L(q,\pm2)$, and any $SU(2)$-representation $\rho$ sends $x^2$ to $1$, so $\rho(x) = \pm 1$.  But then $\rho(x)$ commutes with $\rho(y)$, so $\rho$ has cyclic image.
\end{proof}

Our main tools in the rest of the proof of Theorem~\ref{thm:iterated-cables} will be the following two propositions from \cite{sivek-zentner}.

\begin{proposition}[{\cite[Proposition~10.2]{sivek-zentner}}] \label{prop:sz102}
Let $P$ be a satellite.  If $r$-surgery on $P(K)$ is $SU(2)$-cyclic, then so is $r$-surgery on $P(U)$, where $U$ is the unknot.
\end{proposition}

\begin{proposition}[{\cite[Proposition~10.3]{sivek-zentner}}] \label{prop:sz103}
Let $P$ be a satellite with winding number $w \neq 0$.  If $r$-surgery on $P(K)$ is $SU(2)$-cyclic, then $\frac{r}{w^2}$-surgery on $K$ is also $SU(2)$-cyclic.
\end{proposition}

Let $K = C_{m,n}(T_{p,q})$, where $m\neq 0$ and $n,|p|,q \geq 2$, and $\gcd(m,n)=\gcd(p,q)=1$; and suppose that $r$-surgery on $K$ is $SU(2)$-cyclic for some $r \in \Q$.  Then Proposition~\ref{prop:sz103} says that $\frac{r}{n^2}$-surgery on $T(p,q)$ is $SU(2)$-cyclic as well, which gives a strong restriction on the possible values of $r$.  In particular, by Proposition~\ref{prop:tpq-slopes} we must have $r = n^2(pq+\frac{1}{a})$ or $r=n^2pq$.

\begin{proposition} \label{prop:m-not-1}
If $|m| \neq 1$ and $K = C_{m,n}(T_{p,q})$ admits a nontrivial $SU(2)$-cyclic surgery, then $(m,n)=(2pq\pm 1,2)$ and the $SU(2)$-cyclic surgeries are 
\[ S^3_{mn}(K) = L(2pq\pm1,2q^2) \# \RP^3, \qquad S^3_{mn\mp1}(K) = L(4pq\pm1,4q^2). \]
\end{proposition}

\begin{proof}
The assumption $|m| \neq 1$ tells us that $C_{m,n}(U) = T_{m,n}$ is a nontrivial torus knot.  Then $r$ must be an $SU(2)$-cyclic slope for $T_{m,n}$ as well by Proposition~\ref{prop:sz102}, hence by Proposition~\ref{prop:tpq-slopes} it is either $mn$ or $mn+\frac{1}{b}$ for some integer $b\neq 0$.

If $r=mn$, then we must have either $m=\pm2$ or $n=2$, and
\[ S^3_{mn}(C_{m,n}(T_{p,q})) = S^3_{m/n}(T_{p,q}) \# S^3_{n/m}(U), \]
by \cite[Corollary~7.3]{gordon}.  This is not $SU(2)$-cyclic if $S^3_{m/n}(T_{p,q})$ is not, and $\frac{m}{n}$ is not $pq$ since it is not an integer, so then $\frac{m}{n} = pq + \frac{1}{a} = \frac{apq+1}{a}$ for some nonzero $a \in \Z$.  Both sides are in lowest terms, so $a=\pm n$ and $m=npq\pm1$.  It follows that in fact $n=2$, since $|m| \geq |npq|-1 \geq 11$.  But then $\frac{m}{n}$-surgery on $T_{p,q}$ is the lens space $L(m,nq^2)$ by \cite{moser}, so $S^3_r(K)$ is a connected sum of this lens space with $S^3_{n/m}(U) = \RP^3$.

In the remaining cases, we have $r=mn+\frac{1}{b}$.  If this is equal to $n^2pq$, then $pq = \frac{bmn+1}{bn^2}$, but the right hand side is not an integer since $n \geq 2$ is coprime to $bmn+1$.  Thus we must have
\[ mn+\tfrac{1}{b} = r = n^2\left(pq+\tfrac{1}{a}\right), \]
and so $\frac{bmn+1}{bn^2} - \frac{1}{a} = pq \in \Z$.  Arguing as above we see that $a=\epsilon bn^2$ where $\epsilon=\pm 1$, and then $mn+\frac{1}{b} = n^2pq + \frac{\epsilon}{b}$.  If $\epsilon=1$ then we have $\frac{m}{n}=pq$, which contradicts $\frac{m}{n} \not\in \Z$.  Thus $\epsilon=-1$ and we have $\frac{2}{b} = n(npq-m)$.  The right side is a nonzero integer and a multiple of $n \geq 2$ while the left side has magnitude $\frac{2}{|b|} \leq 2$, so both sides must have magnitude $2$ and we have $b=\mp 1$, $n=2$, $a=-bn^2=\pm 4$, and $m=npq-\frac{1}{b}=2pq\pm1$.  But then $S^3_r(K)$ is
\[ S^3_{mn \mp 1}(C_{m,n}(T_{p,q})) \cong S^3_{(mn\mp1)/n^2}(T_{p,q}) = S^3_{r/n^2}(T_{p,q}) = S^3_{pq\pm\frac{1}{4}}(T_{p,q}), 
\]
where the first homeomorphism comes from \cite[Corollary~7.3]{gordon}, and this is again a lens space by \cite{moser}, namely $L(4pq\pm 1, 4q^2)$.
\end{proof}

Proposition~\ref{prop:m-not-1} handles all cables $K = C_{m,n}(T_{p,q})$ of torus knots except those where $m = \pm 1$.  We will show that $C_{\pm 1,n}(T_{p,q})$ does not have any $SU(2)$-cyclic surgeries at all by making use of the following lemma, many details of which are adapted from the proof of \cite[Theorem~2.8]{ni-zhang}.

\begin{lemma} \label{lem:1-n-arc}
Fix an integer $n \geq 2$, and let $C = C_{1,n}(K)$.  Let $\rho_K: \pi_1(S^3 \ssm K) \to SU(2)$ be an irreducible representation satisfying $\rho_K(\mu_K\lambda_K^n) = 1$ and $\rho_K(\lambda_K) = e^{i\beta}$, where $0 \leq \beta < 2\pi$.  Then there is a continuous family of irreducible representations
\[ \rho_s: \pi_1(S^3 \ssm C) \to SU(2), \qquad 0 \leq s \leq \pi \]
for which $\rho_s(\mu_C^n\lambda_C) = \pm 1$ and $\rho_s(\mu_C)$ is conjugate to $e^{i\alpha_s}$, where
\[ \cos(\alpha_s) = \cos\left(\frac{\pi}{n}\right) \cos(\beta) - \cos(s)\sin\left(\frac{\pi}{n}\right)\sin(\beta). \]
Moreover, we have $\beta \not\in \pi\Z$ and $\alpha_s \not\in \pi\Z$ for all $s$.
\end{lemma}

\begin{proof}
We first observe that $\beta \not\in \pi\Z$.  Indeed, if this were the case then the condition $\rho_K(\mu_K \lambda_K^n) = 1$ would imply that $\rho_K(\mu_K) = \pm 1$, and hence that $\rho_K$ is reducible.  Similarly, once we construct the representations $\rho_s$ it will follow immediately that $\alpha_s \not\in \pi\Z$, since the irreducibility of $\rho_s$ implies $\rho_s(\mu_C) \neq \pm 1$.

We now write $C_{1,n} \subset S^1 \times D^2$ for the pattern knot which produces the $(1,n)$-cable, so that the exterior of $C = C_{1,n}(K)$ decomposes as
\[ S^3 \ssm N(C) = (S^3 \ssm N(K)) \cup_{T^2} ((S^1\times D^2) \ssm N(C_{1,n})). \]
The cable space $X = (S^1 \times D^2) \ssm N(C_{1,n})$ is Seifert fibered, with fiber $\mu_K\lambda_K^n$ on $\partial(S^1\times D^2)$ and $\mu_C^n\lambda_C$ on $\partial N(C_{1,n})$; see \cite[Lemma~7.2]{gordon}.  Here we identify the peripheral curves $\mu_K,\lambda_K \in \partial N(K)$ with the corresponding curves on $\partial (S^1\times D^2)$.

Given the decomposition
\[ \pi_1(S^3 \ssm N(C)) = \pi_1(S^3 \ssm N(K)) \ast_{\pi_1(T^2)} \pi_1(X), \]
we will first define $\rho_s$ on $\pi_1(X)$ and then extend it over $\pi_1(S^3 \ssm N(K))$.  According to the proof of \cite[Claim~2.9]{ni-zhang}, the cable space $X$ has fundamental group
\[ \pi_1(X) = \langle h,\lambda_K \mid h^n\lambda_K = \lambda_Kh^n \rangle, \]
with $\mu_C = h\lambda_K$ being a meridian of $C_{1,n}$.  We define $\rho_s$ on $\pi_1(X)$ by taking $v_s = \cos(s)i+\sin(s)j$ and then setting 
\[ \rho_s(h) = \cos\left(\frac{\pi}{n}\right) + \sin\left(\frac{\pi}{n}\right)v_s, \qquad \rho_s(\lambda_K) = \rho_K(\lambda_K) = e^{i\beta}. \]
Then $\rho_s(h^n) = -1$ commutes with $\rho_s(\lambda_K) = e^{i\beta}$, so $\rho_s$ is well-defined on $\pi_1(X)$.  Moreover, $\rho_s$ is irreducible for $0<s<\pi$ since $\rho_s(h)$ does not commute with $\rho_s(\lambda_K)$ (here we use the fact that $\beta \not\in \pi\Z$), but the Seifert fibers are central in $\pi_1(X)$, so we must have
\[ \rho_s(\mu_K\lambda_K^n) = \rho_s(\mu_C^n\lambda_C) \in \{\pm 1\} \]
for $0<s<\pi$, and then for $0 \leq s \leq \pi$ by continuity.

Now we extend $\rho_s$ to $\pi_1(S^3 \ssm N(K))$ as promised.  We note that $\rho_s(\lambda_K) = e^{i\beta}$ and $\rho_s(\mu_K\lambda_K^n) = \pm 1$ are constant in $s$, and that 
$\rho_s(\lambda_K) = \rho_K(\lambda_K)$ and $\rho_K(\mu_K\lambda_K^n)=1$ by definition.  If in fact $\rho_s(\mu_K\lambda_K^n) = +1$, then $\rho_s|_{\pi_1(X)}$ agrees with $\rho_K$ on their common torus and we let $\rho_s|_{\pi_1(S^3 \ssm N(K))} = \rho_K$.  If instead $\rho_s(\mu_K\lambda_K^n) = -1$, then we replace $\rho_K$ with the representation
\[ \tilde\rho_K = \chi\cdot\rho_K: \pi_1(S^3 \ssm N(K)) \to SU(2), \]
in which we multiply $\rho_K$ by the nontrivial central character
\[ \chi: \pi_1(S^3 \ssm N(K)) \xrightarrow{\operatorname{ab}} H_1(S^3 \ssm N(K)) \cong \Z \xrightarrow{n \mapsto (-1)^n} \{\pm 1\}. \]
Then $\tilde\rho_K(\lambda_K) = \rho_K(\lambda_K)$ and $\tilde\rho_K(\mu_K) = -\rho_K(\mu_K)$, so we have
\[ \tilde\rho_K(\mu_K\lambda_K^n) = -\rho_K(\mu_K\lambda_K^n) = -1 = \rho_s(\mu_K\lambda_K^n) \]
and we let $\rho_s|_{\pi_1(S^3 \ssm N(K))} = \tilde\rho_K$.

The end result of all of this is a family of representations
\[ \rho_s: \pi_1(S^3 \ssm C) \to SU(2), \qquad 0 \leq s \leq \pi, \]
which are irreducible for all $s$ since their restriction to $\pi_1(S^3 \ssm N(K))$ is irreducible, and which satisfy $\rho_s(\mu_C^n\lambda_C) = \pm1$ and
\[ \rho_s(\mu_C) = \rho_s(h\lambda_K) = \left(\cos\left(\frac{\pi}{n}\right)+\sin\left(\frac{\pi}{n}\right)v_s\right) \cdot e^{i\beta}. \]
Thus $\rho_s(\mu_C)$ is conjugate to $e^{i\alpha_s}$ for some $\alpha_s$, and we can compute that $\cos(\alpha_s) = \re(\rho_s(\mu_C))$ is exactly as claimed.
\end{proof}

\begin{proposition} \label{prop:m-1}
If $C=C_{\pm 1,n}(T_{p,q})$ for some $n \geq 2$, then $C$ does not admit any nontrivial $SU(2)$-cyclic surgeries.
\end{proposition}

\begin{proof}
We assume that $C = C_{1,n}(T_{p,q})$, by replacing $C$ with its mirror as needed.

Suppose that $r$ is an $SU(2)$-cyclic slope for $C$.  Just as in Proposition~\ref{prop:m-not-1}, it is still the case that $r$ must be either $n^2pq$ or $n^2(pq+\frac{1}{a})$ for some integer $a\neq 0$, by Proposition~\ref{prop:sz103}.  Note that in particular we must have $|r| \geq (|pq|-1)n^2$.  Our goal will be to show that this is impossible by constructing a path of irreducible representations $\pi_1(S^3 \ssm C) \to SU(2)$ whose image in the pillowcase (i.e., the character variety of the boundary torus) requires $r$ to be close to $n$, by \cite[Proposition~3.1]{sivek-zentner}.

Let $\mu_T,\lambda_T,\mu_C,\lambda_C$ denote meridians and longitudes of $T_{p,q}$ and $C$ respectively.  Proposition~\ref{prop:tpq-slopes} says that $\frac{1}{n}$-surgery on $T_{p,q}$ is not $SU(2)$-cyclic, so there is an irreducible representation
\[ \rho_T: \pi_1(S^3 \ssm T_{p,q}) \to SU(2) \]
such that $\rho_T(\mu_T\lambda_T^n) = 1$, and up to conjugacy we have $\rho_T(\lambda_T) = e^{i\beta}$ for some $\beta \in [0,2\pi)$.  We apply Lemma~\ref{lem:1-n-arc} to see that $\beta \not\in \pi\Z$, and that there is a family of irreducible representations
\[ \rho_s: \pi_1(S^3 \ssm C) \to SU(2), \qquad 0 \leq s \leq \pi \]
such that $\rho_s(\mu_C^n\lambda_C) = \pm 1$ and $\rho_s(\mu_C)$ is conjugate to $e^{i\alpha_s}$, where
\[ \cos(\alpha_s) = \re(\rho_s(\mu_C)) = \cos\left(\frac{\pi}{n}\right) \cos(\beta) - \cos(s)\sin\left(\frac{\pi}{n}\right)\sin(\beta). \]
In particular, we have $\cos(\alpha_0) = \cos\left(\beta+\frac{\pi}{n}\right)$ and $\cos(\alpha_\pi) = \cos\left(\beta-\frac{\pi}{n}\right)$.

We arrange by conjugation that $0 < \alpha_s < \pi$ for all $s$, observing that $\alpha_s \not\in \pi\Z$, and we must then have
\[ \alpha_0 = \begin{cases}
\beta + \tfrac{\pi}{n}, & 0 < \beta < \tfrac{n-1}{n}\pi \\
\tfrac{2n-1}{n}\pi - \beta, & \tfrac{n-1}{n}\pi < \beta < \tfrac{2n-1}{n}\pi \\
\beta - \frac{2n-1}{n}\pi, & \tfrac{2n-1}{n}\pi < \beta < 2\pi
\end{cases},
\qquad
\alpha_\pi = \begin{cases}
\frac{\pi}{n}-\beta, & 0 < \beta < \frac{\pi}{n} \\
\beta-\frac{\pi}{n}, & \frac{\pi}{n} < \beta < \frac{n+1}{n}\pi \\
\frac{2n+1}{n}\pi - \beta, & \frac{n+1}{n}\pi < \beta < 2\pi.
\end{cases} \]
(The cases $\beta = \frac{n-1}{n}\pi, \frac{2n-1}{n}\pi$ and $\beta = \frac{\pi}{n}, \frac{n+1}{n}\pi$ do not occur since they would imply $\alpha_0 \in \pi\Z$ and $\alpha_\pi \in \pi\Z$ respectively.)  Since $\cos(s)$ is monotonic on the interval $0 \leq s \leq \pi$, the same is true of $\cos(\alpha_s)$, and so the values of $\alpha_s$ must cover an interval of length exactly
\begin{equation} \label{eq:cable-alpha}
|\alpha_\pi-\alpha_0| = \begin{cases}
2\beta, & 0 < \beta < \frac{\pi}{n} \\
2|\pi-\beta|, & \pi-\frac{\pi}{n} < \beta < \pi+\frac{\pi}{n} \\
2(2\pi-\beta), & 2\pi-\frac{\pi}{n} < \beta < 2\pi \\
2\pi/n & \mathrm{otherwise}.
\end{cases}
\end{equation}

In the first three cases of \eqref{eq:cable-alpha}, the quantity $|\alpha_\pi-\alpha_0|$ is equal to twice the distance from $\beta$ to $\pi\Z$, so we now seek to bound this quantity.  Since $\rho_T$ is irreducible, it sends the central element $\mu_T^{pq}\lambda_T$ to $\pm 1$, so that $pq\alpha+\beta \equiv 0 \pmod{\pi}$.  Therefore
\[ \begin{pmatrix} pq & 1 \\ 1 & n \end{pmatrix} \begin{pmatrix} \alpha \\ \beta \end{pmatrix} = \begin{pmatrix} \pi k \\ 2\pi l \end{pmatrix} \]
for some integers $k$ and $l$, which implies that $\beta \in \frac{\pi}{pqn-1}\Z$.  We have already seen that $\beta \not\in \pi\Z$, so we have $|\beta - x| \geq \frac{\pi}{|pqn-1|}$ for any $x \in \pi\Z$.  Thus in the first three cases of \eqref{eq:cable-alpha} we have $|\alpha_\pi-\alpha_0| \geq \frac{2\pi}{|pqn-1|}$, and this clearly holds in the fourth case as well.

Now the image of the characters $[\rho_s]$ in the pillowcase is a line segment of slope $-n$, since $\rho_s(\mu_C^n \lambda_C) = \pm 1$ is constant, so if $r=\frac{c}{d}$ is an $SU(2)$-cyclic slope for $C$ then \cite[Proposition~3.1]{sivek-zentner} says that
\[ \left\lvert r - n \right\rvert < \frac{2\pi/|\alpha_\pi-\alpha_0|}{d} \leq \frac{|pqn-1|}{d} \leq |pq|n+1. \]
Recalling that $|r| \geq (|pq|-1)n^2$, we now have
\[ (|pq|-1)n^2 - n \leq |r-n| < |pq|n+1, \]
so $(|pq|-1)n^2 < (|pq|-1)n+(2n+1)$. But if this holds then $n < 1+\frac{2n+1}{(|pq|-1)n} \leq 1+\frac{3n}{5n} < 2$, since $|pq| \geq 6$, and this is a contradiction.
\end{proof}

We can now show that the list of $SU(2)$-cyclic surgeries we have found so far, on torus knots and cables of torus knots, is complete.

\begin{proof}[Proof of Theorem~\ref{thm:iterated-cables}.]
The classification of $SU(2)$-cyclic surgeries on $n$-fold iterated cables with $n \leq 2$ is carried out in Proposition~\ref{prop:tpq-slopes} and Propositions~\ref{prop:m-not-1} and \ref{prop:m-1}.  It thus remains to be shown that there are no $SU(2)$-cyclic surgeries for $n \geq 3$.

Let $K_j = C_{p_j,q_j}(\dots(C_{p_1,q_1}(U))\dots)$ for each $j=1,\dots,n$, and suppose that $r$-surgery on $K=K_n$ is $SU(2)$-cyclic.  Then $\frac{r}{q_n^2}$-surgery on $K_{n-1}$ is $SU(2)$-cyclic by Proposition~\ref{prop:sz103}, so the theorem will follow by induction if we can prove that there are no $SU(2)$-cyclic surgeries when $n=3$.

Write $K = C_{p_3,q_3}(C_{p_2,q_2}(T_{p_1,q_1}))$, where $|p_1| \neq 1$.  If $r$-surgery on $K$ is $SU(2)$-cyclic, then $\frac{r}{q_3^2}$-surgery on $C_{p_2,q_2}(T_{p_1,q_1})$ is $SU(2)$-cyclic (again by Proposition~\ref{prop:sz103}), so the $n=2$ case of the theorem says that
\[ (p_2,q_2) = (2p_1q_1 + \epsilon, 2), \qquad \frac{r}{q_3^2} \in \{ 4p_1q_1 + \epsilon, 4p_1q_1 + 2\epsilon \} \]
for some $\epsilon = \pm 1$.  In particular, $r$ is an integral multiple of $q_3^2$, and $T_{p_2,q_2}$ is a nontrivial torus knot since $p_2 \neq \pm 1$.

Applying Proposition~\ref{prop:sz102} to the satellite $P(L) = C_{p_3,q_3}(C_{p_2,q_2}(L))$, we see that $r$ is also an $SU(2)$-cyclic slope for $P(U) = C_{p_3,q_3}(T_{p_2,q_2})$.  Since $T_{p_2,q_2}$ is a nontrivial torus knot, we apply the $n=2$ case of the theorem to see that
\[ (p_3,q_3) = (2p_2q_2 + \epsilon', 2), \qquad r \in \{4p_2q_2 + \epsilon', 4p_2q_2 + 2\epsilon'\} \]
for some $\epsilon' = \pm 1$.  Since $q_3 = 2$, it follows from above that $r$ is a multiple of 4, but here we have shown that $r$ is either $\pm 1$ or $\pm 2$ modulo $4$, and so we have a contradiction.
\end{proof}

\section{Dehn surgery and the spliced torus knot complements} \label{sec:elementary-obstruction}

We wish to study the question of which of the $SU(2)$-cyclic manifolds
\[ Y(T_{a,b},T_{c,d}) = E_{a,b} \cup_{T^2} E_{c,d} \]
can be realized by Dehn surgery on a knot in $S^3$.  Here we recall that $E_{a,b}$ and $E_{c,d}$ are the exteriors of the torus knots $T_{a,b}$ and $T_{c,d}$, and the gluing map takes the meridian and Seifert fiber of $E_{a,b}$ to the Seifert fiber and meridian of $E_{c,d}$ respectively.

\subsection{Non-integral surgeries}

First, we will show that Eudave-Mu\~noz's list of half-integral, toroidal surgeries on hyperbolic knots in $S^3$ contains all of the $Y(T_{a,b},T_{c,d})$ which arise as non-integral Dehn surgery on some knot.  

\begin{theorem} \label{thm:y-nonintegral-surgery}
Let $K \subset S^3$ be a knot on which some Dehn surgery of non-integral slope $r$ produces $Y(T_{a,b}, T_{c,d})$.  Then $K$ is an Eudave-Mu\~noz knot and 
\[ S^3_r(K) \cong \pm Y(T_{l,lm-1},T_{2,-(2m-1)}) \]
for some $l$ and $m$.
\end{theorem}

\begin{proof}
Let $Y = Y(T_{a,b},T_{c,d})$.  We begin by noting that $K$ cannot be a torus knot or even an iterated torus knot, since by Proposition~\ref{prop:tpq-slopes} and Theorem~\ref{thm:iterated-cables} the $SU(2)$-cyclic surgeries on such $K$ are all lens spaces or connected sums of lens spaces, hence atoroidal.

If $K$ is hyperbolic, then since $Y$ is toroidal, Gordon and Luecke \cite{gordon-luecke-nonintegral} proved that $K$ is one of the Eudave-Mu\~noz knots and $r$ is the corresponding half-integral slope.  In this case $S^3_r(K)$ has the desired form by Theorem~\ref{thm:em-classification}, whose second and third cases are in fact the same as the first up to mirroring $K$ and reversing the orientation of the $r$-surgery manifold.

We may therefore assume from now on that $K$ is a nontrivial satellite.  We will write $K = P(L)$, where $P \subset S^1 \times D^2$ is the pattern and $L \subset S^3$ the companion.  Then $T = \partial(S^3 \ssm N(L))$ is an incompressible, non-boundary-parallel torus in the exterior of $K$.  We will consider two cases separately: either $T$ remains incompressible in $Y = S^3_r(K)$, or it does not.

First, suppose that $T$ is incompressible in $Y = S^3_r(K)$.  Lemma~\ref{lem:splicing-tori} says that it must then be isotopic in $Y$ to $\partial E_{a,b} = \partial E_{c,d}$.  The components of its complement are $S^3 \ssm N(L)$ and the $r$-surgery on $P \subset S^1 \times D^2$, and these must be Seifert fibered since they are identified in some order with the two components of $Y \ssm T$.  The first implies that $L$ is a torus knot \cite{moser}; for the second, a theorem of Miyazaki and Motegi \cite[Theorem~1.2]{miyazaki-motegi-3} says that since $r$ is neither an integer nor $\infty$ and $P$ is not a core of $S^1 \times D^2$, $P$ must be a cable of a 0-bridge braid.  In particular, $K=P(L)$ is an iterated cable of $L$, hence an iterated torus knot, and we have already ruled this out.

Now we know that $T = \partial(S^3 \ssm N(L))$ can be compressed in $Y$.  Then $T$ compresses when the exterior of $K$ is filled with slope either $r$ or the meridian $\mu$, so \cite[Theorem~2.0.1]{cgls} says that either $\Delta(\mu,r) \leq 1$, or $T$ and $\partial(S^3 \ssm N(K))$ cobound a cable space.  The first possibility does not apply since $r$ is non-integral, so $K$ must be a cable of $L$, say $K = C_{p,q}(L)$ with $\gcd(p,q)=1$ and $q \geq 2$, and then $S^3_r(K)$ is described by \cite[Corollary~7.3]{gordon}.  Namely, we have $r \neq pq$ since $r \not\in \Z$; and we cannot have $\Delta(r,pq) > 1$ since then $S^3_r(K)$ would be the union of $S^3 \ssm N(L)$ and a Seifert fiber space with incompressible boundary, in which $T = \partial(S^3 \ssm N(L))$ would be incompressible after all.  Thus $\Delta(r,pq) = 1$, and we have $S^3_{r/q^2}(L) = S^3_r(K) = Y$.

We write $r = pq + \frac{1}{n}$ for some nonzero integer $n$, and note that $n\neq \pm 1$ since $r \not\in \Z$.  The slope $\frac{r}{q^2}$ of the $Y$-surgery on $L$ is $\frac{npq+1}{nq^2}$ in lowest terms, so its denominator has magnitude $|n|q^2 \geq 8$.  This means that $L$ cannot be an Eudave-Mu\~noz knot, so again it is a satellite of some $L'$.  The same argument as above tells us that $L = C_{p',q'}(L')$, where $p',q'$ are relatively prime integers with $q' \geq 2$, and that $\Delta(\frac{r}{q^2}, p'q') = 1$.  Now from $r=pq+\frac{1}{n}$ and $\Delta(\frac{r}{q^2},p'q')=1$ we have an integer $k \neq 0$ such that
\[ \frac{npq+1}{nq^2} = \frac{r}{q^2} = p'q' + \frac{1}{k} = \frac{kp'q'+1}{k}. \]
Both sides are in lowest terms, so $k=\pm nq^2$.  If $k=nq^2$, then this simplifies to $\frac{p}{q} = p'q'$, which is impossible since $\frac{p}{q}\not\in\Z$.  Thus $k=-nq^2$, and multiplying both sides by $nq^2$ gives $npq+1 = nq^2p'q'-1$, or $nq(qp'q'-p) = 2$.  But since $|n|,|q| \geq 2$ we conclude that $2$ is a multiple of $|nq| \geq 4$, and this is a contradiction.
\end{proof}

\begin{remark}
The proof of Theorem~\ref{thm:y-nonintegral-surgery} only used three facts about $Y=Y(T_{a,b},T_{c,d})$: 
\begin{enumerate}
\item $Y$ has a unique incompressible torus up to isotopy,
\item this torus splits $Y$ into two Seifert fibered spaces, and
\item $Y$ is $SU(2)$-cyclic.
\end{enumerate}
We can thus also conclude that if $Y'$ is the half-integral, toroidal $r$-surgery on the Eudave-Mu\~noz knot $K=k(l,m,0,p)$, and if $2p-1$ divides $l$, then $Y'$ can only arise from $r'$-surgery on some knot $K'$ if either $r'$ is integral or $K'$ is an Eudave-Mu\~noz knot and $r'$ is its half-integral, toroidal slope.  Here we apply \cite[Lemma~4.4]{ni-zhang} for the uniqueness of the incompressible torus in $Y'$, and Theorem~\ref{thm:em-classification} to show that $Y'$ is $SU(2)$-cyclic.
\end{remark}

\subsection{Integral surgeries}

Theorem~\ref{thm:y-nonintegral-surgery} asserts that $Y = Y(T_{a,b},T_{c,d})$ cannot result from non-integral Dehn surgery on a knot in $S^3$ unless $Y \cong \pm Y(T_{l,lm-1},T_{2,-(2m-1)})$ for some integers $l$ and $m$.  Most $Y$ do not have this form, and in this section we will show that many of them cannot be produced by integral surgery either, proving that they cannot arise from Dehn surgery on a knot in $S^3$ at all.

Letting $Y=Y(T_{a,b},T_{c,d})$, we recall from Proposition~\ref{prop:linking-form} that
\[ \ell_Y(\mu_{a,b}, \mu_{a,b}) = -\frac{cd}{abcd-1} \qquad\mathrm{and}\qquad \ell_Y(\mu_{c,d}, \mu_{c,d}) = -\frac{ab}{abcd-1}. \]
As another example of a linking form computation, let $n$ be a nonzero integer, and suppose that $Y$ is realized by $n$-surgery on a knot $K \subset S^3$.  Then $H_1(Y)$ is generated by the meridian $\mu_K$, and $Y$ is formed by gluing a solid torus to $S^3 \ssm N(K)$ so that $n\mu_K + \lambda_K$ bounds a disk $D$ in the solid torus.  In particular, $n\mu_K$ bounds the 2-chain $D - \Sigma$, where $\Sigma \subset S^3 \ssm N(K)$ is a Seifert surface for $K$, and so $\ell_Y(\mu_K,\mu_K) = \big((D-\Sigma)\cdot \mu_K\big)/n = -\frac{1}{n}$.

\begin{proposition} \label{prop:ab-residue}
Let $Y=Y(T_{a,b},T_{c,d})$.
\begin{itemize}
\item If $ab$ and $cd$ are not perfect squares modulo $|abcd-1|$, then $Y$ is not the result of $(abcd-1)$-surgery on any knot in $S^3$.
\item If $-ab$ and $-cd$ are not perfect squares modulo $|abcd-1|$, then $Y$ is not the result of $-(abcd-1)$-surgery on any knot in $S^3$.
\end{itemize}
\end{proposition}

\begin{proof}
We prove only the first assertion.  If $Y$ arises from $(abcd-1)$-surgery on a knot $K$, then it has linking form $\ell_Y(x\mu_K,x\mu_K) = -\frac{x^2}{abcd-1}$ for all $x \in \Z/(abcd-1)\Z$.  Taking $x$ so that $x\mu_K = \mu_{a,b}$ (resp.\ $\mu_{c,d}$), Proposition~\ref{prop:linking-form} says that $x^2 \equiv cd$ (resp.\ $ab$) modulo $|abcd-1|$, so $ab$ and $cd$ must be perfect squares modulo $|abcd-1|$, a contradiction.
\end{proof}

\begin{remark}
We observe that $\pm ab$ has inverse $\pm cd$ in $\Z/|abcd-1|\Z$, so $\pm ab$ is a perfect square modulo $|abcd-1|$ if and only if $\pm cd$ is.  Thus in Proposition~\ref{prop:ab-residue} it suffices to check whether $ab$ and $-ab$ are perfect squares modulo $|abcd-1|$.
\end{remark}

As an example, Proposition~\ref{prop:ab-residue} implies the following.
\begin{proposition} \label{prop:p-1-mod-8}
Let $Y = Y(T_{a,b},T_{-a,b})$ for some coprime integers $a,b \geq 2$.  If some odd prime divisor of $(ab)^2+1$ is not 1 modulo 8, then $Y$ is not the result of integral surgery on any knot in $S^3$.
\end{proposition}

\begin{proof}
Write $n=ab$ for convenience.  The slope of such a surgery would have to be $\pm|H_1(Y)| = \pm(n^2+1)$, so by Proposition~\ref{prop:ab-residue}, it suffices to show that $n$ and $-n$ are not perfect squares modulo $n^2+1$.

Let $p$ be an odd prime dividing $n^2+1$, and suppose that $p \not\equiv 1 \pmod{8}$.  If we can find $x$ such that $x^2 \equiv \pm n \pmod{n^2+1}$, then $x^4 \equiv n^2 \equiv -1 \pmod{n^2+1}$ and hence $x^4 \equiv -1 \pmod{p}$.  It follows that $x^8 \equiv 1 \pmod{p}$ and hence that $x$ has order 8 in the cyclic group $\F_p^\times$ of order $p-1$, so $p\equiv 1 \pmod{8}$, which is a contradiction.
\end{proof}

\begin{theorem} \label{thm:ab-ab}
Let $T_{a,b}$ be a nontrivial torus knot, with $a,b > 2$.  If $ab$ does not belong to the set
\[ S = \{n \in \N \mid p\equiv1\!\!\!\!\pmod{8} \textrm{ for all odd primes } p \mid n^2+1\}, \]
which has density zero in $\N$, then $Y(T_{a,b},T_{-a,b})$ is not Dehn surgery on any knot in $S^3$.
\end{theorem}

\begin{proof}
The condition $a,b > 2$ guarantees by Theorem~\ref{thm:y-nonintegral-surgery} that $Y = Y(T_{a,b},T_{-a,b})$ cannot arise from non-integral surgery on a knot in $S^3$.  If $ab \not\in S$ then Proposition~\ref{prop:p-1-mod-8} rules out integral surgeries as well, so it remains to be shown that $S$ has density zero.

If some integer $n$ belongs to $S$, then we must have $n^2 \not\equiv -1 \pmod{p}$ for every prime $p \equiv 5\pmod{8}$.  Letting $p_1,\dots,p_k$ be the first $k$ primes congruent to $5\pmod{8}$, we define
\[ S_k = \{n \geq 1 \mid n^2 \not\equiv -1\!\!\!\pmod{p_i} \textrm{ for }i=1,\dots,k\} \]
and note that $S \subset S_k$.  Since $-1$ is a quadratic residue (and thus has exactly two square roots) modulo each $p_i$, it is straightforward to show that 
\begin{equation} \label{eq:s_k-density}
\lim_{n\to\infty} \frac{|S_k \cap \{1,2,\dots,n\}|}{n} = \prod_{i=1}^k \left(1 - \frac{2}{p_i}\right) < \exp\left(-\sum_{i=1}^k \frac{2}{p_i}\right).
\end{equation}
Indeed, if we let $N = p_1p_2\dots p_k$ and write $mN \leq n < (m+1)N$ for some integer $m$, then
\begin{align*}
\frac{|S_k \cap \{1,2,\dots,n\}|}{n} &\geq \frac{|S_k \cap \{1,2,\dots,mN\}|}{(m+1)N} \\
&= \frac{m\prod_{i=1}^k (p_i-2)}{(m+1)N} = \frac{m}{m+1} \prod_{i=1}^k \left(1-\frac{2}{p_i}\right), 
\end{align*}
and similarly
\[ \frac{|S_k \cap \{1,2,\dots,n\}|}{n} \leq \frac{|S_k \cap \{1,2,\dots,(m+1)N\}|}{mN} \leq \frac{m+1}{m} \prod_{i=1}^k\left(1-\frac{2}{p_i}\right), \]
establishing the left side of \eqref{eq:s_k-density}.  The right side of \eqref{eq:s_k-density} then comes from the inequality $1-x < e^{-x}$ for all $x > 0$.  In particular, since $S \subset S_k$ it follows that
\[ 0 \leq \limsup_{n\to\infty} \frac{|S \cap \{1,2,\dots,n\}|}{n} \leq \exp\left(-\sum_{i=1}^k \frac{2}{p_i}\right), \]
and the sum $\sum_{i=1}^k \frac{2}{p_i}$ goes to infinity as $k$ does, since the sum $\sum \frac{1}{p}$ over all primes is unbounded and (by Dirichlet's theorem) the primes $p\equiv 5\pmod{8}$ have density $\frac{1}{4}$ in the set of all primes.  We conclude that the desired limit exists and is zero.
\end{proof}

\begin{theorem} \label{thm:ab+ab}
Let $T_{a,b}$ be a nontrivial torus knot.  If $|ab|$ does not belong to the set
\[S' = \{n \in \N \mid p\not\equiv 3\!\!\!\!\pmod{4} \textrm{ for all primes } p \mid n-1 \textrm{ \textit{or} for all } p \mid n+1\}, \]
of density zero in $\N$, then $Y(T_{a,b},T_{a,b})$ is not the result of Dehn surgery on any knot in $S^3$.
\end{theorem}

\begin{proof}
Theorem~\ref{thm:y-nonintegral-surgery} says that $Y = Y(T_{a,b},T_{a,b})$ cannot arise from non-integral surgery.  By Proposition~\ref{prop:ab-residue}, if it is an integral surgery on some knot, then either $ab$ or $-ab$ must be a perfect square modulo $(ab)^2-1$.  

Let $n=|ab|$.  If $n$ is a square modulo $n^2-1$, then $-1 \equiv n \pmod{n+1}$ is a quadratic residue, hence $-1$ is a square modulo every prime $p$ dividing $n+1$ and so no such $p$ can be $3\pmod{4}$.  Similarly, if $-n$ is a square modulo $n^2-1$, then $-1 \equiv -n \pmod{n-1}$ is a quadratic residue and so no prime $p \mid n-1$ can be $3\pmod{4}$.  Thus if $Y(T_{a,b},T_{a,b})$ is integral surgery on some knot then $n \in S'$, as claimed.

It remains to be seen that $S'$ has density zero in $\N$.  For this, it suffices to prove that the set
\[ T = \{n \in \N \mid p \not\equiv 3\!\!\!\!\!\pmod{4} \textrm{ for all primes } p \mid n\} \]
has density zero, since then
\[ S' = \{ n \geq 2 \mid n-1 \in T \textrm{ or } n+1 \in T \} \]
satisfies
\[ \frac{|S' \cap \{1,2\dots,n\}|}{n} \leq \frac{2|T \cap \{1,2,\dots,n+1\}|}{n} \to 0 \]
as $n \to \infty$.

We now follow the same reasoning as in the proof of Theorem~\ref{thm:ab-ab}.  We let $p_1,\dots,p_k$ denote the first $k$ primes which are congruent to $3\pmod{4}$, and let $T_k \subset \N$ consist of all integers which are not multiples of any of $p_1,\dots,p_k$.  Then 
\[ \lim_{n\to\infty} \frac{|T_k \cap \{1,2,\dots,n\}|}{n} = \prod_{i=1}^k \left(1-\frac{1}{p_i}\right) < \exp\left(-\sum_{i=1}^k \frac{1}{p_i}\right) \]
just as before.  Since $T \subset T_k$ and $\sum_{i=1}^k \frac{1}{p_i}$ is unbounded as $k \to \infty$, we conclude again that $T$ must have density zero.
\end{proof}

\section{Changemakers and integral surgeries} \label{sec:changemakers}

We now consider a different approach to the question of when $Y(T_{a,b},T_{c,d})$ is the result of integral surgery on some knot in $S^3$.  Our motivation comes from the family of $3$-manifolds
\[ Y=Y(T_{2,2a+1},T_{2,2b+1}), \qquad a,b \geq 1, \]
whose first homology has order $n = 4(2a+1)(2b+1)-1$.  We can show using linking forms that $Y$ is never $n$-surgery on a knot in $S^3$ -- this is the content of Proposition~\ref{prop:2m-square} below, in combination with Proposition~\ref{prop:ab-residue} -- but these do not suffice to rule out $Y$ as $-n$-surgery on some knot, and so something stronger is needed.  In fact, we will prove the following.

\begin{theorem} \label{thm:2-2a+1-2-2b+1}
Suppose for some positive integers $a \leq b$ that $Y = Y(T_{2,2a+1},T_{2,2b+1})$ is Dehn surgery on a knot in $S^3$.  Then 
\[ (a,b) \in \{ (1,1),(1,2),(1,3),(2,3),(3,3) \}, \]
and the slope is $-n$ where $n = |H_1(Y;\Z)| = 4(2a+1)(2b+1)-1$.
\end{theorem}

\begin{proof}
Theorem~\ref{thm:y-nonintegral-surgery} says that $Y$ cannot be the result of non-integral surgery, so we need only rule out surgeries of slope $\pm n$.  We will show in Proposition~\ref{prop:2m-square} that $2(2a+1)$ is not a quadratic residue modulo $4(2a+1)(2b+1)-1$, so Proposition~\ref{prop:ab-residue} says that $Y$ cannot be $n$-surgery on a knot either.  Then in \S\ref{ssec:changemakers} we will review Greene's changemaker obstruction \cite{greene} and use it to prove Proposition~\ref{prop:22-negative-integer}, asserting that $Y$ cannot be $-n$-surgery either except possibly in the five cases above.
\end{proof}

\subsection{Positive integral surgeries} \label{ssec:residue-prop}

In this subsection we prove the following number-theoretic result, which is used in the proof of Theorem~\ref{thm:2-2a+1-2-2b+1} with $m=2a+1$ and $n=2b+1$ to rule out positive integral surgeries.

\begin{proposition} \label{prop:2m-square}
Let $m,n \geq 1$ be odd integers.  Then $2m$ is not a perfect square modulo $4mn-1$.
\end{proposition}

Throughout this subsection we use $\legendre{a}{p}$ to denote the Legendre symbol, which for any prime number $p$ and integer $a$ coprime to $p$ is equal to $+1$ if $a$ is a square $\pmod{p}$ and $-1$ otherwise.  For odd primes $p$ and $q$ we also define
\[ \epsilon_{p,q} = (-1)^{\frac{p-1}{2} \cdot \frac{q-1}{2}}, \]
so that the law of quadratic reciprocity states that $\legendre{p}{q} = \epsilon_{p,q} \legendre{q}{p}$.

\begin{lemma} \label{lem:2m-legendre-symbol}
Let $p,q$ be primes which do not divide $2m$, where $m$ is odd, and which satisfy $p\equiv q \pmod{8m}$.  Then $\legendre{2m}{p} = \legendre{2m}{q}$.
\end{lemma}

\begin{proof}
We factor $2m = 2\cdot a_1^{e_1} \dots a_k^{e_k}$, where the $a_i$ are distinct odd primes and $e_i \geq 1$ for all $i$.  Then
\[ \legendre{2m}{p} = \legendre{2}{p} \prod_{i=1}^k \legendre{a_i}{p}^{e_i} = (-1)^{\frac{p^2-1}{8}} \prod_{i=1}^k \left[ \epsilon_{a_i,p} \legendre{p}{a_i} \right]^{e_i} \]
by quadratic reciprocity, and likewise for $q$.  Since $p$ and $q$ are congruent mod 8, we have $(-1)^{(p^2-1)/8} = (-1)^{(q^2-1)/8}$ and also $\epsilon_{a_i,p}=\epsilon_{a_i,q}$ for all $i$.  In addition, we have $p \equiv q \pmod{a_i}$ and hence $\legendre{p}{a_i} = \legendre{q}{a_i}$ for all $i$, since $a_i$ divides $8m$.  We conclude that
\[ \legendre{2m}{p} = (-1)^{\frac{q^2-1}{8}} \prod_{i=1}^k \left[\epsilon_{a_i,q}\legendre{q}{a_i}\right]^{e_i} = \legendre{2m}{q}. \qedhere \]
\end{proof}

Each invertible residue class modulo $8m$ contains infinitely many primes, by Dirichlet's theorem on primes in arithmetic progressions, so we define a character
\[ \chi_{8m}: (\Z/8m\Z)^\times \to \{\pm 1\} \]
by setting $\chi_{8m}(a) = \legendre{2m}{p}$ where $p$ is any prime congruent to $a \pmod{8m}$.  Lemma~\ref{lem:2m-legendre-symbol} asserts that $\chi_{8m}$ is well-defined, and the following lemma asserts that it is indeed a character.

\begin{lemma} \label{lem:chi-8m-character}
The function $\chi_{8m}$ is a group homomorphism.
\end{lemma}

\begin{proof}
Fix primes $p$ and $q$ not dividing $2m$ and take a third prime $r \equiv pq \pmod{8m}$.  We will show that $\legendre{\ell}{r} = \legendre{\ell}{p} \legendre{\ell}{q}$ for every prime $\ell$ dividing $2m$.  Then if we write $2m = 2a_1^{e_1} \cdots a_k^{e_k}$ as before, it follows that
\begin{align*}
\legendre{2m}{r} &= \legendre{2}{r}\prod_{i=1}^k \legendre{a_i}{r}^{e_i} \\
&= \legendre{2}{p}\legendre{2}{q} \prod_{i=1}^k \left[\legendre{a_i}{p}\legendre{a_i}{q}\right]^{e_i} = \legendre{2m}{p} \legendre{2m}{q},
\end{align*}
so that $\chi_{8m}(r) = \chi_{8m}(p) \chi_{8m}(q)$.

The claim that $\legendre{2}{r} = \legendre{2}{p} \legendre{2}{q}$ follows immediately from the fact that $\legendre{2}{p}$ is $1$ if and only if $p \equiv \pm1\pmod{8}$, and likewise for $q$ and $r \equiv pq\pmod{8}$.  We can thus assume now that $\ell$ is odd, and several applications of quadratic reciprocity, plus the fact that $r\equiv pq\pmod{\ell}$ since $\ell$ divides $8m$, give 
\[ \legendre{\ell}{r} = \epsilon_{\ell,r}\legendre{r}{\ell} = \epsilon_{\ell,r} \legendre{pq}{\ell} = \epsilon_{\ell,r} \legendre{p}{\ell}\legendre{q}{\ell} = \epsilon_{\ell,r}\epsilon_{\ell,p}\epsilon_{\ell,q} \legendre{\ell}{p}\legendre{\ell}{q}. \]
We therefore wish to show that $\epsilon_{\ell,r}\epsilon_{\ell,p}\epsilon_{\ell,q} = 1$.  The left side is
\begin{align*} (-1)^{\frac{l-1}{2}\cdot \frac{r-1}{2}} (-1)^{\frac{l-1}{2}\cdot\frac{p-1}{2}} (-1)^{\frac{l-1}{2}\cdot\frac{q-1}{2}} &= (-1)^{\frac{l-1}{2}\cdot \frac{r+p+q-3}{2}} \\
&= (-1)^{\frac{l-1}{2}\left(\frac{pq+p+q+1}{2} - 2\right)} \\
&= (-1)^{\frac{l-1}{2} \cdot \left(\frac{(p+1)(q+1)}{2}-2\right)}.
\end{align*}
The exponent is even since $(p+1)(q+1)$ is a multiple of 4, so the product $\epsilon_{\ell,r}\epsilon_{\ell,p}\epsilon_{\ell,q}$ is indeed equal to 1, as desired.
\end{proof}

\begin{lemma} \label{lem:chi-4m-1}
We have $\chi_{8m}(4m-1) = -1$.
\end{lemma}

\begin{proof}
Let $p$ be a prime congruent to $4m-1\pmod{8m}$; we wish to show that $\legendre{2m}{p} = -1$.  Writing $2m = 2\cdot a_1^{e_1} \dots a_k^{e_k}$ once again, with each $a_i$ an odd prime, we claim that $\legendre{a_i}{p} = 1$ for all $i$.  Indeed, we observe that $p\equiv 3\pmod{4}$ and that $p+1$ is a multiple of $m$ and hence of $a_i$, so quadratic reciprocity gives us
\[ \legendre{a_i}{p} = \epsilon_{a_i,p}\legendre{p}{a_i} = (-1)^{\frac{a_i-1}{2}} \legendre{-1}{a_i} = 1. \]
We can therefore conclude that
\[ \legendre{2m}{p} = \legendre{2}{p} \prod_{i=1}^k \legendre{a_i}{p}^{e_i} = \legendre{2}{p} = -1, \]
where the last equality comes from the fact that $p \equiv 4m-1 \equiv 3\pmod{8}$.
\end{proof}

We now use the character $\chi_{8m}$ to prove Proposition~\ref{prop:2m-square}.

\begin{proof}[Proof of Proposition~\ref{prop:2m-square}]
Suppose that $2m$ is a square modulo $4mn-1$.  Then it is a square modulo every prime $p$ which divides $4mn-1$, so $\chi_{8m}(p) = 1$ for all such $p$.  (We note that $\gcd(4mn-1,8m)=1$, since it divides $n(8m) - 2(4mn-1) = 2$ and since $4mn-1$ is odd.)  In particular, the residue class of $4mn-1$  modulo $8m$ is a product of elements of $\ker(\chi_{8m})$, so it also belongs to this kernel.  But since $n$ is odd we have $4mn-1 \equiv 4m-1\pmod{8m}$, so we have just shown that $\chi_{8m}(4m-1) = 1$, and this contradicts Lemma~\ref{lem:chi-4m-1}.
\end{proof}

\subsection{Changemakers and negative surgeries} \label{ssec:changemakers}

We first recall some facts about alternating links and their branched double covers.  Given a reduced, non-split, alternating diagram $D$ of a link $L$, we produce a checkerboard coloring of the plane by alternately coloring regions black and white according to the following convention:
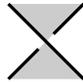
\[ \begin{tikzpicture}[very thick]
    \fill[fill=black!20] (-0.5,0.5)--(0.5,0.5)--(-0.5,-0.5)--(0.5,-0.5);
    \draw (0.1,0.1) -- (0.5,0.5);
    \draw (0.5,-0.5) -- (-0.5,0.5);
    \draw (-0.5,-0.5) -- (-0.1,-0.1);
\end{tikzpicture} \]
The white graph $W_D$ has one vertex per white region, and for each crossing of $D$ we add an edge between the vertices of $W_D$ corresponding to the adjacent white regions.  If we arbitrarily number the vertices of $W_D$ as $v_0,v_1,\dots,v_n$, then the corresponding \emph{Goeritz matrix} is the $n\times n$ matrix whose $(i,i)$th entry is $-\deg(v_i)$ and whose $(i,j)$th entry for $i\neq j$ is the number of edges between $v_i$ and $v_j$, where $1 \leq i, j \leq n$.

Now since $L$ is alternating, its branched double cover $\Sigma(L)$ bounds a compact, negative definite 4-manifold $X_L$ with $H_1(X_L)=0$ and intersection form $Q_{X_L}$ given by the Goeritz matrix $M$ corresponding to an alternating diagram, by \cite[\S 3]{gordon-litherland-signature}.  (It follows that $M$ also gives a presentation of $H_1(\Sigma(L))$.)  Moreover, $\Sigma(L)$ is a Heegaard Floer L-space and $X_L$ is \emph{sharp}, by \cite[Proposition~3.3 and Theorem~3.4]{osz-branched}.%
\footnote{The Goeritz form in \cite{osz-branched} is described equivalently in terms of the black graph $B_D$, being generated by edges of $B_D$ after removing a maximal tree $T$; here we count generators to see that these are in bijection with vertices of $W_D$.  Supposing $D$ has $c$ crossings, and viewing it as a planar 4-valent graph, the formula $V-E+F=2$ becomes $c-2c+(\#V(W_D)+\#V(B_D))=2$.  Since $B_D$ and $T$ have $c$ and $\#V(B_D)-1$ edges respectively, their Goeritz form then has $c-(\#V(B_D)-1) = \#V(W_D)-1$ generators.}
Sharpness is a technical condition which we will not define here (see e.g.\ \cite[Definition~2.1]{greene}), but which implies a very strong consequence proved by Greene.

\begin{definition} \label{def:changemaker}
A vector $\sigma = (\sigma_0, \sigma_1, \dots, \sigma_n)$ with nonnegative integer entries $0 \leq \sigma_0 \leq \sigma_1 \leq \dots \leq \sigma_n$ is called a \emph{changemaker} if it satisfies
\[ \sigma_i \leq \sigma_0 + \dots + \sigma_{i-1} + 1 \]
for all $i = 0,1,\dots,n$.
\end{definition}

Let $\Lambda_{-p}$ denote the rank-1 lattice whose generator $\sigma$ satisfies $\langle\sigma,\sigma\rangle = -p$.

\begin{theorem}[{\cite[Theorem~3.3]{greene}}] \label{thm:changemaker}
Let $K$ be a knot in $S^3$ which admits a positive L-space surgery, and let $p$ be a positive integer such that $Y = S^3_p(K)$ bounds a sharp 4-manifold $X$ with $b_2(X) = n$.  Then there is a lattice embedding 
\[ H_2(X) \oplus \Lambda_{-p} \hookrightarrow -\Z^{n+1}, \]
whose image has full rank and where the generator $\sigma$ of $\Lambda_{-p}$ maps to a changemaker with $\langle\sigma,\sigma\rangle = -p$.
\end{theorem}

We observe that replacing an alternating diagram for $L$ with its mirror exchanges the roles of the black and white graphs, so the Goeritz matrix constructed from the black graph (rather than the white graph) is the intersection form of a sharp 4-manifold $X_{\overline{L}}$ bounded by $\Sigma(\overline{L}) = -\Sigma(L)$.  In particular, if $\Sigma(L)$ is $(-p)$-surgery on a knot $K \subset S^3$ for some integer $p > 0$, then applying Theorem~\ref{thm:changemaker} to the identification $-\Sigma(L) \cong S^3_p(\overline{K})$ gives a lattice embedding $H_2(X_{\overline{L}}) \oplus \Lambda_{-p} \hookrightarrow -\Z^{b_2(X_{\overline{L}})+1}$.

In \cite[\S 6]{zentner-simple}, the second author constructed alternating links $L=L(T_{a,b},T_{c,d})$ whose branched double covers are the manifolds $Y(T_{a,b},T_{c,d})$; these are knots if $abcd-1$ is odd and links otherwise.  Thus the $Y(T_{a,b},T_{c,d})$ are L-spaces and we can apply Theorem~\ref{thm:changemaker} to $X_L$ to study whether they can be realized by positive integral surgeries, or to $X_{\mirror{L}}$ for negative integral surgeries.

Figure~\ref{fig:LT35T-35} illustrates this construction for the 2-component link $L(T_{3,5},T_{-3,5})$; the Conway sphere which lifts to an incompressible torus is not visible in the alternating diagram.  We remark that $Y=Y(T_{3,5},T_{-3,5})$ is not realized by non-integral surgery on any knot in $S^3$, by Theorem~\ref{thm:y-nonintegral-surgery}, but that the linking form obstruction of Proposition~\ref{prop:p-1-mod-8} does not suffice to rule out integral surgery, since the only odd prime divisor of
\[ (3\cdot 5)^2+1 = 2\cdot 113 \]
is 1 modulo 8.  Here we can restrict our attention to positive integral surgeries, since if $Y \cong S^3_{-226}(K)$ then we have $Y \cong -Y \cong S^3_{226}(\mirror{K})$.

\tikzset{link/.style = { white, double = black, line width = 1.75pt, double distance = 1.25pt, looseness=1.75 }}
\tikzset{blacklabel/.style={draw, fill=black!20, font=\tiny, circle, inner sep = 0.075cm}}
\begin{figure}
\begin{tikzpicture}[scale=0.75]
\begin{scope}
\clip (-2.15,-3.15) rectangle (5.85,3.6);
\coordinate (A) at (45:2);
\coordinate (B) at (135:2);
\coordinate (C) at (225:2);
\coordinate (D) at (315:2);

\draw[link] (A) to[out=225,in=-45] (-0.25,0.75) to[out=135,in=90] (-1.25,0);
\coordinate (E) at (1.2,-0.75);
\draw[link] (-1.25,0) to[out=270,in=225] (-0.25,-0.75) to[out=45,in=90] (E);
\draw[link] (E) to[out=270,in=315] (-0.25,-0.75) to[out=135,in=45] (C);
\coordinate (F) at (0.8,0.8);
\draw[link,looseness=1] (D) to[out=135,in=270] (0.75,-0.75) to[out=90,in=270] (1.25,0) to[out=90,in=315] (F);
\draw[link] (F) to[out=135,in=45] (-0.25,0.75) to[out=225,in=315] (B);
\begin{scope} \clip (2,0) -- (2,1.5) -- (-2,1.5) -- (-2,-1.5) -- (-0.5,-1.5) -- (-0.5,1) -- (0,1) -- (0,0) -- cycle;
\draw[link] (A) to[out=225,in=-45] (-0.25,0.75) to[out=135,in=90] (-1.25,0) to[out=270,in=225] (-0.25,-0.75);
\end{scope}
\begin{scope} \clip (0.5,0) rectangle (1.5,-1);
\draw[link] (-0.25,-0.75) to[out=45,in=90] (E);
\end{scope}

\draw[link] (A) to[out=45,in=135,looseness=1] (3.25,1.25) to[out=315,in=45] (3.25,0.25) to[out=225,in=180] (3.25,-1.25) to[out=0,in=315] (3.25,0.25) to[out=135,in=225] (3.25,1.25);
\draw[link] (3.25,1.25) to[out=45,in=210] (2.75,2.5) to[out=30,in=150] (3.75,2.5) to[out=330,in=0,looseness=1.5] (1,-3) to[out=180,in=225,looseness=1.25] (C);
\draw[link] (B) to[out=135,in=150,looseness=1.25] (2.75,2.5) to[out=330,in=210] (3.75,2.5);
\draw[link] (D) to[out=315,in=225,looseness=1] (2.75,-1.25) to[out=45,in=135] (3.75,-1.25) to[out=315,in=30,looseness=2.25] (3.75,2.5);
\begin{scope} \clip (3.25,2.25) rectangle (4.25,2.75);
\draw[link] (2.75,2.5) to[out=30,in=150] (3.75,2.5) to[out=330,in=0,looseness=1.5] (1,-3);
\end{scope}
\begin{scope} \clip (2.75,1.5) rectangle (3.75,1);
\draw[link] (A) to[out=45,in=135,looseness=1] (3.25,1.25) to[out=315,in=45] (3.25,0.25);
\end{scope}
\begin{scope} \clip (2.75,0) rectangle(3.75,0.5);
\draw[link] (3.25,1.25) to[out=225,in=135] (3.25,0.25) to[out=315,in=0] (3.25,-1.25);
\end{scope}
\begin{scope} \clip (2.25,-1.5) rectangle (3.25,-0.75);
\draw[link] (3.25,0.25) to[out=225,in=180] (3.25,-1.25);
\end{scope}
\draw[thick,dotted] (0,0) circle (2); 
\end{scope}
\begin{scope}[xshift=9.5cm]
\clip (-3,-3.15) rectangle (5.85,3.6);
\coordinate (A) at (45:1);
\coordinate (B) at (135:2);
\coordinate (C) at (225:2);
\coordinate (D) at (315:2);

\coordinate (E) at (1.2,-0.75);
\draw[link] (E) to[out=270,in=315] (-0.25,-0.75) to[out=135,in=45] (C);
\coordinate (F) at (0.8,0.8);
\draw[link,looseness=1] (D) to[out=135,in=270] (0.75,-0.75) to[out=90,in=270] (1.25,0) to[out=90,in=315] (F);
\draw[link] (F) to[out=135,in=45] (-0.25,0.75) to[out=225,in=315] (B);
\draw[link] (A) to[out=225,in=-45] (-0.25,0.75) to[out=135,in=180] (2,2) to[out=0,in=0, looseness=2.25] (2,-2) to[out=180,in=225] (-0.25,-0.75) to[out=45,in=90] (E);
\begin{scope} \clip (-0.75,0.5) rectangle (0.25,1);
\draw[link] (F) to[out=135,in=45] (-0.25,0.75) to[out=225,in=315] (B);
\end{scope}
\begin{scope} \clip (0.5,0) rectangle (1.5,-1);
\draw[link] (-0.25,-0.75) to[out=45,in=90] (E);
\end{scope}

\draw[link] (A) to[out=45,in=165,looseness=1] (2,1.25) to[out=345,in=180] (3.25,-1.25) to[out=0,in=315] (3.25,0.25) to[out=135,in=0,looseness=1] (0,0.1) to[out=180,in=150,looseness=3.5] (3.75,2.5) (3.75,2.5) to[out=330,in=0,looseness=1.5] (1,-3) to[out=180,in=225,looseness=1.25] (C);
\draw[link] (B) to[out=135,in=150,looseness=1.25] (2.75,2.5) to[out=330,in=210] (3.75,2.5);
\draw[link] (D) to[out=315,in=225,looseness=1] (2.75,-1.25) to[out=45,in=135] (3.75,-1.25) to[out=315,in=30,looseness=2.25] (3.75,2.5);
\begin{scope} \clip (-1.5,0) rectangle (1.5,0.5);
\draw[link,looseness=1] (0.75,-0.75) to[out=90,in=270] (1.25,0) to[out=90,in=315] (F);
\end{scope}
\begin{scope} \clip (3.25,2.25) rectangle (4.25,2.75);
\draw[link] (0,0.1) to[out=180,in=150,looseness=3.5] (3.75,2.5) to[out=330,in=0,looseness=1.5] (1,-3);
\end{scope}
\begin{scope} \clip (2.5,0) rectangle (3.75,0.7);
\draw[link] (0,0.1) to[out=0,in=135,looseness=1] (3.25,0.25) to[out=315,in=0] (3.25,-1.25);
\end{scope}
\begin{scope} \clip (2.25,-1.5) rectangle (3.25,-0.75);
\draw[link] (2,1.25) to[out=345,in=180] (3.25,-1.25);
\end{scope}
\begin{scope} \clip (3.5,-1) rectangle (4.5,-2);
\draw[link] (2,2) to[out=0,in=0, looseness=2.25] (2,-2);
\end{scope}
\begin{scope} \clip (-0.75,-1.25) rectangle (0.25,-0.25);
\draw[link] (E) to[out=270,in=315] (-0.25,-0.75) to[out=135,in=45] (C);
\end{scope}
\node[blacklabel] at (4.75,2.25) {$0$};
\node[blacklabel] at (-2.1,1.25) {$1$};
\node[blacklabel] at (4,0.3) {$2$};
\node[blacklabel] at (1.85,-0.5) {$3$};
\node[blacklabel] at (0.35,-0.75) {$4$};
\node[blacklabel] at (1,-2.5) {$5$};
\draw[dashed] (4,-2.5) -- (3.25,-1.1);
\node[blacklabel] at (4,-2.5) {$6$};
\end{scope}
\end{tikzpicture}
\caption{Left, a diagram for $L(T_{3,5},T_{-3,5})$ with dotted circle indicating a Conway sphere.  Right, an alternating diagram for the same link with regions of the white graph labeled.} \label{fig:LT35T-35}
\end{figure}
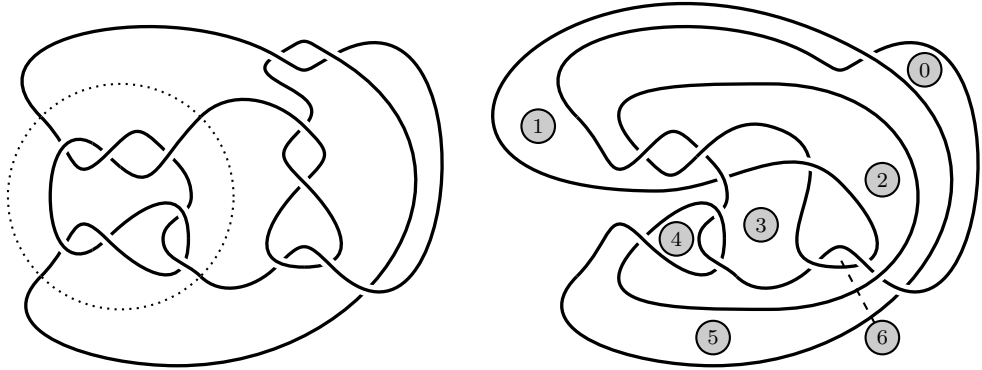

Examining the white graph of the alternating diagram $D$ in Figure~\ref{fig:LT35T-35}, we see that $L=L(T_{3,5},T_{-3,5})$ has Goeritz matrix
\[ G_D = \begin{pmatrix}
-4 & 2 & 1 & 0 & 0 & 0 \\
2 & -5 & 1 & 0 & 1 & 1 \\
1 & 1 & -5 & 2 & 0 & 1 \\
0 & 0 & 2 & -3 & 1 & 0 \\
0 & 1 & 0 & 1 & -3 & 0 \\
0 & 1 & 1 & 0 & 0 & -2
\end{pmatrix}, \]
which is the matrix of the intersection form on the handlebody $X_L$ with boundary $Y \cong \Sigma(L)$.
\begin{proposition} \label{prop:changemaker-35}
There is a lattice embedding $H_2(X_L) \oplus \Lambda_{-226} \hookrightarrow -\Z^7$ which sends the generator of $\Lambda_{-226}$ to the changemaker
\[ \sigma = (1,2,2,4,4,8,11). \]
\end{proposition}
\begin{proof}
The sublattice of $-\ZZ^7$ with basis $(v_1,v_2,\dots,v_6)$ consisting of the vectors
\[ \begin{pmatrix}1\\0\\0\\0\\-1\\-1\\1\end{pmatrix},
\begin{pmatrix}-2\\0\\1\\0\\0\\0\\0\end{pmatrix},
\begin{pmatrix}1\\0\\1\\1\\1\\0\\-1\end{pmatrix},
\begin{pmatrix}0\\0\\0\\-1\\-1\\1\\0\end{pmatrix},
\begin{pmatrix}0\\-1\\-1\\1\\0\\0\\0\end{pmatrix},
\begin{pmatrix}0\\1\\-1\\0\\0\\0\\0\end{pmatrix} \]
lies in the orthogonal complement $\sigma^\perp$, and its Gram matrix $\begin{pmatrix} \langle v_i,v_j\rangle \end{pmatrix}$ is $G_D$.
\end{proof}

Proposition~\ref{prop:changemaker-35} says that the changemaker obstruction is also not strong enough to determine whether $Y=Y(T_{3,5},T_{-3,5})$ can be realized by integral surgery on a knot.  In fact, the changemaker $\sigma$ given above is the unique one satisfying the conclusion of Theorem~\ref{thm:changemaker}.  This uniquely determines the Alexander polynomial of any knot on which $226$-surgery produces $Y$, as explained in \cite[\S 2]{greene}; here we merely point out \cite[Proposition~3.1]{greene}, by which such a knot must have genus $\frac{1}{2}(226-|\sigma|_1) = 97$.

\begin{question} \label{q:L35}
Can $Y(T_{3,5},T_{-3,5})$ be realized by $226$-surgery on a knot in $S^3$?
\end{question}

After the initial version of this paper was made public, Duncan McCoy informed us that Question~\ref{q:L35} has a positive answer.

\begin{proposition}[McCoy] \label{prop:L35}
The manifold $Y(T_{3,5},T_{-3,5})$ can be realized as $226$-surgery on a knot in $S^3$.
\end{proposition}

\begin{proof}
In Figure~\ref{fig:mccoy-band}, we have illustrated a band which may be attached to turn the link $L(T_{3,5},T_{-3,5})$ into an unknot.  This operation corresponds to performing surgery on the branched double cover $Y=Y(T_{3,5},T_{-3,5})$ to get the branched double cover of the unknot, which is $S^3$; turning this around, we realize $Y$ by surgery on a knot $K$ in $S^3$, and so $Y$ must arise from a $226$-surgery by the above discussion.
\begin{figure}
\begin{tikzpicture}
\begin{scope}
\clip (-3,-3.15) rectangle (5.85,3.6);
\coordinate (A) at (45:1);
\coordinate (B) at (135:2);
\coordinate (C) at (225:2);
\coordinate (D) at (315:2);

\coordinate (E) at (1.2,-0.75);
\draw[link] (E) to[out=270,in=315] (-0.25,-0.75) to[out=135,in=45] (C);
\coordinate (F) at (0.8,0.8);
\draw[link,looseness=1] (D) to[out=135,in=270] (0.75,-0.75) to[out=90,in=270] (1.25,0) to[out=90,in=315] (F);
\draw[link] (F) to[out=135,in=45] (-0.25,0.75) to[out=225,in=315] (B);
\draw[link] (A) to[out=225,in=-45] (-0.25,0.75) to[out=135,in=180] (2,2) to[out=0,in=0, looseness=2.25] (2,-2) to[out=180,in=225] (-0.25,-0.75) to[out=45,in=90] (E);
\begin{scope} \clip (-0.75,0.5) rectangle (0.25,1);
\draw[link] (F) to[out=135,in=45] (-0.25,0.75) to[out=225,in=315] (B);
\end{scope}
\begin{scope} \clip (0.5,0) rectangle (1.5,-1);
\draw[link] (-0.25,-0.75) to[out=45,in=90] (E);
\end{scope}

\draw[link] (A) to[out=45,in=165,looseness=1] (2,1.25) to[out=345,in=180] (3.25,-1.25) to[out=0,in=315] (3.25,0.25) to[out=135,in=0,looseness=1] (0,0.1) to[out=180,in=150,looseness=3.5] (3.75,2.5) (3.75,2.5) to[out=330,in=0,looseness=1.5] (1,-3) to[out=180,in=225,looseness=1.25] (C);
\draw[link] (B) to[out=135,in=150,looseness=1.25] (2.75,2.5) to[out=330,in=210] (3.75,2.5);
\draw[link] (D) to[out=315,in=225,looseness=1] (2.75,-1.25) to[out=45,in=135] (3.75,-1.25) to[out=315,in=30,looseness=2.25] (3.75,2.5);
\begin{scope} \clip (-1.5,0) rectangle (1.5,0.5);
\draw[link,looseness=1] (0.75,-0.75) to[out=90,in=270] (1.25,0) to[out=90,in=315] (F);
\end{scope}
\begin{scope} \clip (3.25,2.25) rectangle (4.25,2.75);
\draw[link] (0,0.1) to[out=180,in=150,looseness=3.5] (3.75,2.5) to[out=330,in=0,looseness=1.5] (1,-3);
\end{scope}
\begin{scope} \clip (2.5,0) rectangle (3.75,0.7);
\draw[link] (0,0.1) to[out=0,in=135,looseness=1] (3.25,0.25) to[out=315,in=0] (3.25,-1.25);
\end{scope}
\begin{scope} \clip (2.25,-1.5) rectangle (3.25,-0.75);
\draw[link] (2,1.25) to[out=345,in=180] (3.25,-1.25);
\end{scope}
\begin{scope} \clip (3.5,-1) rectangle (4.5,-2);
\draw[link] (2,2) to[out=0,in=0, looseness=2.25] (2,-2);
\end{scope}
\begin{scope} \clip (-0.75,-1.25) rectangle (0.25,-0.25);
\draw[link] (E) to[out=270,in=315] (-0.25,-0.75) to[out=135,in=45] (C);
\end{scope}

\begin{scope}[lband/.style = { white, double = black, line width = 0.75pt, double distance = 1pt, looseness=1.25 }]
\begin{scope}
\clip (0.61,0.97) to[out=60,in=180,looseness=0.9] (1.85,1.75) to[out=170,in=55,looseness=1.25] (0.4,1.1) -- cycle;
\draw[fill=gray!30] (0,0) rectangle (3,2);
\end{scope}
\begin{scope}
\clip (3.3,0.2) -- (3.1,0.4) to[out=90,in=-10] (1.86,1.77) to[out=10,in=80,looseness=1.2] (3.3,0.2);
\draw[fill=gray!30] (1.5,0) rectangle (3.5,2);
\end{scope}
\draw[lband,looseness=1] (0.6,0.95) to[out=60,in=180] (2.25,1.8) to[out=0,in=80] (3.3,0.2);
\draw[lband,looseness=1] (0.4,1.1) to[out=60,in=180] (1.5,1.8) to[out=0,in=90] (3.1,0.35);
\end{scope}
\begin{scope}[lfix/.style = { white, double = black, line width = 0.5pt, double distance = 1.25pt, looseness=1.75 }]
\begin{scope}
\clip (0.25,0.9) rectangle (0.75,1.25);
\draw[lfix] (F) to[out=135,in=45] (-0.25,0.75) to[out=225,in=315] (B);
\end{scope}
\begin{scope}
\clip (2.9,0.1) rectangle (3.5,0.55);
\draw[lfix] (0,0.1) to[out=0,in=135,looseness=1] (3.25,0.25) to[out=315,in=0] (3.25,-1.25);
\end{scope}
\end{scope}
\end{scope}
\end{tikzpicture}
\caption{A band which turns the link $L(T_{3,5},T_{-3,5})$ into an unknot.}
\label{fig:mccoy-band}
\end{figure}
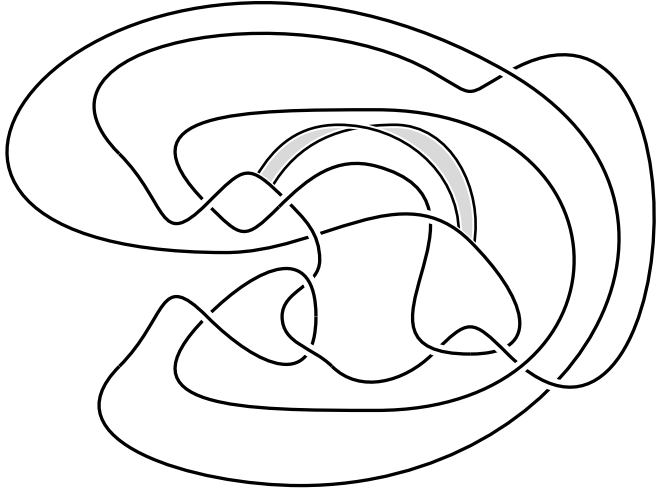
\end{proof}

For some more conclusive applications of the changemaker technique, Figure~\ref{fig:LTpqTrs} shows a family of knots $L(T_{p,q},T_{r,s})$, as constructed in \cite[\S 6]{zentner-simple}, whose branched double covers are $Y(T_{p,q},T_{r,s})$.  Here we restrict to the special case
\[ \frac{p}{q} = a_0 + \frac{1}{a_1}, \qquad \frac{r}{s} = b_0 + \frac{1}{b_1} \]
where $a_0,a_1,b_0,b_1$ are positive integers, and in fact $a_1,b_1 \geq 2$ since $q$ and $s$ should not be $1$.  In the figure we have labeled some of the black regions; the twist regions labeled $1-b_1$ and $1-a_1$ contain an additional $b_1-2$ and $a_1-2$ black regions, respectively, so that the black graph has $a_1+b_1+2$ vertices.  We will use some general facts about changemakers, combined with a handful of computer-assisted computations, to show that many of the $Y(T_{p,q},T_{r,s})$ cannot be Dehn surgeries on knots in $S^3$.
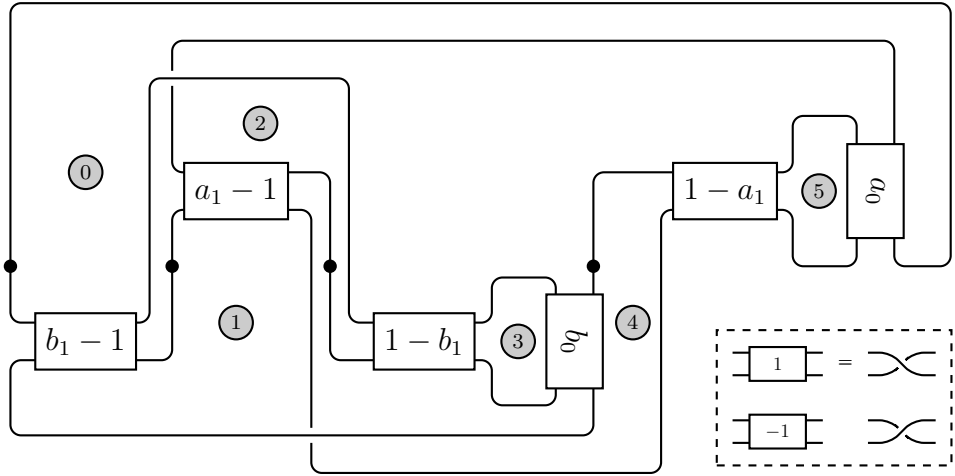
\begin{figure}
\tikzset{twistregion/.style={draw, fill=white, thick, minimum width=1.25cm, minimum height=0.75cm}}
\tikzset{blacklabel/.style={draw, fill=black!20, font=\tiny, circle, inner sep = 0.075cm}}
\tikzset{dot/.style={draw,fill,circle,inner sep=0pt,minimum size=0.15cm}}
\[ \begin{tikzpicture}[thick]
\draw [rounded corners] (-0.5,0.25) -- (-1.5,0.25) -- (-1.5,4.5) -- (11,4.5) -- (11,1) -- (10.25,1) -- (10.25,2);
\draw [rounded corners] (-0.5,-0.25) -- (-1.5,-0.25) -- (-1.5,-1.25) -- (6.25,-1.25) -- (6.25,0);
\draw [rounded corners] (-0.5,-0.25) -- (0.65,-0.25) -- (0.65,1.75) -- (1.5,1.75);
\draw [rounded corners] (-0.5,0.25) -- (0.35,0.25) -- (0.35,3.5) -- (3,3.5) -- (3,0.25) -- (4,0.25);
\draw [rounded corners] (1.5,2.25) -- (2.75,2.25) -- (2.75,-0.25) -- (4,-0.25);
\draw [rounded corners] (1.5,2.25) -- (0.65,2.25) -- (0.65, 3.4) (0.65,3.6) -- (0.65, 4) -- (10.25,4) -- (10.25,2);
\draw [rounded corners] (1.5,1.75) -- (2.5,1.75) -- (2.5,-1.15) (2.5,-1.35) -- (2.5,-1.75) -- (7.15,-1.75) -- (7.15,1.75) -- (8,1.75);
\draw [rounded corners] (4,-0.25) -- (4.9, -0.25) -- (4.9,-0.85) -- (5.75, -0.85) -- (5.75,0);
\draw [rounded corners] (4,0.25) -- (4.9,0.25) -- (4.9,0.85) -- (5.75, 0.85) -- (5.75,0);
\draw [rounded corners] (6.25,0) -- (6.25, 2.25) -- (8,2.25);
\draw [rounded corners] (8,2.25) -- (8.9,2.25) -- (8.9,3) -- (9.75,3) -- (9.75,2);
\draw [rounded corners] (8,1.75) -- (8.9,1.75) -- (8.9,1) -- (9.75,1) -- (9.75,2);
\node [twistregion] at (-0.5,0) {$b_1-1$}; 
\node [twistregion] at (1.5,2) {$a_1-1$};
\node [twistregion] at (4,0) {$1-b_1$};
\node [twistregion] at (6,0) [rotate=-90] {$b_0$};
\node [twistregion] at (8,2) {$1-a_1$};
\node [twistregion] at (10,2) [rotate=-90] {$a_0$};
\node [dot] at (-1.5,1) {};
\node [dot] at (0.65,1) {};
\node [dot] at (2.75,1) {};
\node [dot] at (6.25,1) {};
\node [blacklabel] at (-0.5,2.25) {$0$};
\node [blacklabel] at (1.5,0.25) {$1$};
\node [blacklabel] at (1.825,2.9) {$2$};
\node [blacklabel] at (5.25,0) {$3$};
\node [blacklabel] at (6.775,0.25) {$4$};
\node [blacklabel] at (9.25,2) {$5$};
\begin{scope}[scale=0.6, xshift=5.5cm, yshift=-0.5cm, transform shape]
\draw [dashed] (7.65,0.75) rectangle (12.85, -2.25);
\draw (8,0.25) -- (10,0.25) (8,-0.25) -- (10,-0.25);
\node [twistregion] at (9,0) {$1$};
\node at (10.5,0) {$=$};
\draw [rounded corners] (11,0.25) -- (11.5,0.25) -- (12,-0.25) -- (12.5,-0.25) (11,-0.25) -- (11.5,-0.25) -- (11.6,-0.15) (11.9,0.15) -- (12,0.25) -- (12.5,0.25);
\draw (8,-1.25) -- (10,-1.25) (8,-1.75) -- (10,-1.75);
\node [twistregion] at (9,-1.5) {$-1$};
\node at (10.5,0) {$=$};
\draw [rounded corners] (11,-1.25) -- (11.5,-1.25) -- (11.6, -1.35) (11.9,-1.65) -- (12,-1.75) -- (12.5,-1.75) (11,-1.75) -- (11.5,-1.75) -- (12,-1.25) -- (12.5,-1.25);
\end{scope}
\end{tikzpicture} \]
\caption{The knot $L(T_{p,q},T_{r,s})$ whose branched double cover is $Y=Y(T_{p,q},T_{r,s})$.  Here $\frac{p}{q} = a_0 + \frac{1}{a_1}$ and $\frac{r}{s} = b_0 + \frac{1}{b_1}$ for some integers $a_0,a_1,b_0,b_1 > 0$.  The black dots indicate points of intersection with a Conway sphere, which lifts to an incompressible torus in $Y$, and the circled numbers indicate some vertices of the black graph.}
\label{fig:LTpqTrs}
\end{figure}

\begin{lemma} \label{lem:changemaker-bound}
Let $\sigma = (\sigma_0,\dots,\sigma_n)$ be a changemaker.  Then $|\langle \sigma,\sigma \rangle| \leq \frac{1}{3}(4^{n+1}-1)$.
\end{lemma}

\begin{proof}
An easy induction shows that $\sigma_i \leq 2^i$, since if it holds for all $j < i$ then
\[ \sigma_i \leq \sigma_0 + \dots + \sigma_{i-1}+1 \leq (1+2+\dots+2^{i-1})+1 = 2^i. \]
Thus $|\langle \sigma,\sigma \rangle|$ is at most $1 + 4 + 4^2 + \dots + 4^n$.
\end{proof}

\begin{proposition} \label{prop:goeritz-bound}
Let $D$ be a non-split, alternating diagram of a link $L$, and let $p = |H_1(\Sigma(L);\Z)|$.  If $\Sigma(L)$ is $p$-surgery on some knot $K \subset S^3$, then the white graph $W_D$ has at least $\log_4(3p+1)$ vertices.  If it is $-p$-surgery on some knot, then the black graph $B_D$ has at least $\log_4(3p+1)$ vertices.
\end{proposition}

\begin{proof}
Let $v=n+1$ be the number of vertices in $W_D$, and suppose that $\Sigma(L)$ arises from a $p$-surgery.  We apply Theorem~\ref{thm:changemaker} to the sharp $4$-manifold $X_L$ with intersection form given by the $n\times n$ Goeritz matrix $G_D$.  This gives us a changemaker $\sigma$ with $\langle\sigma,\sigma\rangle = -p$, so that
\[ p = |\langle \sigma,\sigma\rangle| \leq \tfrac{1}{3}(4^{n+1}-1) \]
by Lemma~\ref{lem:changemaker-bound}.  Thus $4^v \geq 3p+1$, or $v \geq \log_4(3p+1)$ as claimed.

If instead $\Sigma(L)$ arises as $-p$-surgery, then we repeat the same argument with $B_D$, whose Goeritz matrix describes the intersection form on the sharp manifold $X_{\mirror{L}}$ bounded by $\Sigma(\mirror{L})=-\Sigma(L)$, in place of $W_D$ to get the desired bound.
\end{proof}

\begin{proposition} \label{prop:22-negative-integer}
Let $Y = Y(T_{2,2a+1},T_{2,2b+1})$, where $a$ and $b$ are positive integers and $a \leq b$, and let
\[ n = |H_1(Y;\Z)| = 4(2a+1)(2b+1)-1. \]
If $Y$ is $-n$-surgery on some knot in $S^3$, then
\[ (a,b) \in \{ (1,1),(1,2),(1,3),(2,3),(3,3) \}. \]
\end{proposition}

\begin{proof}
Supposing that $Y$ is $-n$-surgery on a knot, we obtain an alternating diagram $D$ of $L(T_{2a+1,2},T_{2b+1,2})$ from Figure~\ref{fig:LTpqTrs} by setting $(a_0,a_1,b_0,b_1)=(a,2,b,2)$.  In this case, the six numbered regions in Figure~\ref{fig:LTpqTrs} are all of the vertices of the black graph $B_D$, so Proposition~\ref{prop:goeritz-bound} tells us that
\[ \log_4(3n+1) \leq 6 \quad\Longleftrightarrow\quad n \leq 1365, \]
or $(2a+1)(2b+1) \leq \frac{1366}{4} = 341.5$.

For the remaining cases, we compute that the black graph for the given diagram of $L(T_{2a+1,2},T_{2b+1,2})$ has Goeritz matrix
\[ \begin{pmatrix}
-3 & 1 & 0 & 1 & 0 \\ 
1 & -3 & 1 & 0 & 0 \\ 
0 & 1 & -b-1 & b & 0 \\ 
1 & 0 & b & -b-2 & 1 \\ 
0 & 0 & 0 & 1 & -a-1 
\end{pmatrix}. \]
A computer search among the cases where $1 \leq a \leq b$ and $(2a+1)(2b+1) \leq 341$ shows that the corresponding lattice does not embed in the orthogonal complement of a changemaker unless
\[ (a,b) \in \{ (1,1),(1,2),(1,3),(2,3),(3,3) \}, \]
and hence $Y=Y(T_{2a+1,2},T_{2b+1,2})$ does not arise from a $-n$-surgery except possibly in these five cases.
\end{proof}

This completes the proof of Theorem~\ref{thm:2-2a+1-2-2b+1}. \hfill\qed

\bibliographystyle{myalpha}
\bibliography{References}

\end{document}